\documentclass[a4paper,11pt, leqno]{article}
\usepackage[nospace,noadjust]{cite}
\usepackage{authblk}
\usepackage{mathrsfs}
\usepackage{graphics}
\usepackage{latexsym}
\usepackage{graphicx,psfrag}
\usepackage{amsmath,amscd,amssymb,amsthm,mathtools,mathabx}
\usepackage{comment}
\usepackage{xcolor}
\usepackage{bbm}
\usepackage{filecontents}
\usepackage[center, bf, small]{titlesec}
\usepackage{relsize}

\usepackage{indentfirst}
\titlelabel{\thetitle.\quad}
\makeatletter
\NewCommandCopy\latexparagraph\paragraph
\RenewDocumentCommand{\paragraph}{sO{#3}m}{%
  \IfBooleanTF{#1}
    {\latexparagraph*{\maybe@addperiod{#3}}}
    {\latexparagraph[#2]{\maybe@addperiod{#3}}}%
}
\newcommand{\maybe@addperiod}[1]{%
  #1\@addpunct{.}%
}
\makeatother

\newtheorem*{proposition*}{Proposition}

\newtheorem{theorem}{Theorem}\numberwithin{theorem}{section}
\newtheorem*{main}{Main~Theorem}
\newtheorem{lemma}{Lemma}\numberwithin{lemma}{section}
\newtheorem{proposition}{Proposition}\numberwithin{proposition}{section}
\numberwithin{corollary}{section}
\numberwithin{definition}{section}
\numberwithin{remark}{section}
\newtheorem*{notation}{Notation}
\numberwithin{example}{section}
\newtheorem{ob}{Observation}\numberwithin{ob}{section}
\numberwithin{nt}{section}
\numberwithin{equation}{section}

\newcommand{\RN}[1]{%
  \textup{\uppercase\expandafter{\romannumeral#1}}%
}

\newcommand{\N}{\mathbb{N}}
\newcommand{\Z}{\mathbb{Z}}

\newcommand{\R}{\mathbb{R}}

\newcommand{\MS}{\mathcal{S}}

\newcommand{\MA}{\mathcal{A}}
\newcommand{\MM}{\mathcal{M}}

\newcommand{\bH}{\mathbb{H}}

\newcommand{\MV}{\mathcal{V}}
\newcommand{\uc}{\bar{c}}

\newcommand{\uxi}{\bar{\xi}}

\newcommand{\ux}{\bar{x}}

\newcommand{\uy}{\bar{y}}

\newcommand{\uz}{\bar{z}}

\newcommand{\sm}{d\sigma}
\newcommand{\pd}{\partial}
{}

\newcommand{\pose}[1]{\left(#1\right)}
\newcommand{\llvert}{\left\lvert}
\newcommand{\rrvert}{\right\rvert}
\newcommand{\rrVert}{\right\rVert}
\newcommand{\llVert}{\left\lVert}

\newcommand{\dotprod}[2]{\left\langle #1 \, , \, #2 \right\rangle}

\title{Restricted weak type endpoint estimate for the spherical maximal operators on the Heisenberg group}
\author[1]{Hyunwoo Jeon}
\author[1]{Joonil Kim}
\affil[1]{Department of Mathematics, Yonsei University, Seoul 120-749, Korea}

\begin{document}
\maketitle 
\begin{abstract}
Let $\bH^n$ denote the Heisenberg group, identified with $\mathbb{R}^d \times \mathbb{R}$, where $d = 2n$ and $n \in \mathbb{N}$. We consider the spherical maximal operator $\MM$ associated with the sphere $S^{d-1}$ embedded in the horizontal subspace $\mathbb{R}^d \times \{0\}$ of $\bH^n$. It is known that $\MM$ is bounded on $L^p(\bH^n)$ if and only if $p \in \bigl(\tfrac{d}{d-1}, \infty\bigr]$. In this paper, we establish a restricted weak type $(p,p)$ estimate at the endpoint $p = \tfrac{d}{d-1}$ for $\MM$, provided $d \ge 3$.
\end{abstract}

\noindent
{\bf Keywords:} Spherical maximal average, Heisenberg group

\smallskip \noindent{\small {{MSC2010: 42B15, 42B30} }}

\section{Introduction}
In this paper, we shall establish a restricted weak type estimate at the endpoint $p=\frac{d}{d-1}$ for the spherical maximal operator $\MM$ on the Heisenberg group $\bH^n$. To set the stage, let $\pose{\bH^n,\cdot}$ denote the Heisenberg group, identified with $\R^d\times \R$, where $d=2n\in\N$. We write a general point in $\bH^n$ as $\ux=\pose{x,x_{d+1}}\in\R^d\times\R$, and the group law $\cdot$ is given by 
\[\ux\cdot\uy=\pose{x,x_{d+1}}\cdot\pose{y,y_{d+1}}=\pose{x+y,x_{d+1}+y_{d+1}+\langle x,Jy\rangle},\]
where $J$ is an invertible skew-symmetric $d\times d$ matrix.

Throughout the paper, the averaging operators we consider are defined via the noncommutaitve convolution
\begin{align}
f\ast_{\bH}K\pose{\ux}&=\int f\pose{\uy^{-1}\cdot\ux}K\pose{\uy}d\uy\label{H-convolution}\\
&=\int f\pose{x-y,x_{d+1}-y_{d+1}+\langle x,Jy\rangle}K\pose{y,y_{d+1}}dydy_{d+1}.\nonumber
\end{align}

Let $\sm$ be the normalized surface measure on the $\pose{d-1}$-dimensional unit sphere in the hoizontal subspace $\R^d\times\left\{0\right\}$, centered at the origin. For $t>0$, define the dilate $\sm_t$ by 
\[\left\langle f,\sm_t\right\rangle=\int f\pose{tx,0}\sm\pose{x}.\]  
Given a Schwartz function $f$ on $\bH^n\cong\R^{d}\times\R$, we define the spherical averages by
\begin{equation}\label{average}
\MA\pose{f}\pose{\ux,t}:=f\ast_{\bH}\sm_t\pose{\ux}=\int_{S^{d-1}}f\pose{x-ty,x_{d+1}-t\langle x,Jy\rangle}\sm\pose{y},
\end{equation}
where $S^{d-1}$ is the unit sphere in $\R^d$. The corresponding maximal operator \(\MM\) is then given by 
\begin{equation}\label{maximaldf}
\MM\pose{f}\pose{\ux}=\sup_{t>0}\llvert \MA\pose{f}\pose{\ux,t}\rrvert.
\end{equation}

In \cite{Nevo}, Nevo and Thangavelu initiated the study of the maximal operator $\MM$ on the Heisenberg group $\bH^n = \mathbb{R}^d \times \mathbb{R}$, with $d = 2n \ge 4$. They proved that $\MM$ is bounded on $L^p$ for $p > \tfrac{d-1}{d-2}$. An optimal $L^p$ boundedness result for $\MM$ on $\bH^n$ (with $n \ge 2$), namely that $\MM$ is bounded on $L^p$ if and only if $p > \tfrac{d}{d-1}$ was proved by M\"uller and Seeger \cite{seeger} and, independently and by a different method, by Narayanan and Thangavelu \cite{NT2}. In \cite{seeger}, they actually cover a larger class of groups, namely the M\'{e}tivier groups, which strictly contains the class of groups of Heisenberg-type. For the case of $n=1$, the $L^p$ boundedness of $\MM$ for $p>2$ remains open; however, positive results have been obtained for Heisenberg radial functions (see \cite{Leecircular} and \cite{Guocircular}). On the other hand, an interesting phenomenon arises when $J$ is not a skew-symmetric matrix. In \cite{JoonilKim}, the second author investigated more general matrices according to the multiplicities of their eigenvalues.

More recently, Ryu and Seeger \cite{ryu} established the sharp range $p > \tfrac{d}{d-1}$ for a broader class of groups, namely two-step nilpotent Lie groups of the form $\mathbb{R}^d \times \mathbb{R}^m$ with $d \ge 3$ and $m \ge 1$. They also obtained a restricted weak type inequality at the endpoint $p = \tfrac{d}{d-1}$ for the \emph{local} maximal operator
\[
\MM_{loc}\pose{f}\pose{\tilde{x}}= \sup_{t \in I} \llvert\MA\pose{f}\pose{\tilde{x},t}\rrvert,
\]
where $I\subset[0,\infty)$ is a compact interval and $\tilde{x}\in\R^d\times\R^m$. In the same paper, they posed questions regarding the case of \emph{global} $t$. To date, restricted weak type estimates at $p = \tfrac{d}{d-1}$ for global $t$ remain unproved, even for the Heisenberg group. The aim of this paper is to prove a restricted weak type estimate at the endpoint $p=\frac{d}{d-1}$ for the global maximal operator $\MM$ defined in \eqref{maximaldf} on the Heisenberg group $\R^d\times\R$ with $d>2$.

\begin{main}\label{maintheorem} 
Let $d > 2$ and let $\mathcal{M}$ be the spherical maximal operator defined in \eqref{maximaldf}. Then $\mathcal{M}$ is of restricted weak type $(p,p)$ for $p = \frac{d}{d-1}$:
for every finite measurable set $E \subset \mathbb{R}^{d+1}$ (with characteristic function $1_E$), there exists a constant $C > 0$, independent of $E$, such that
\begin{equation}
\left\| \mathcal{M}(1_{E}) \right\|_{L^{p,\infty}(\mathbb{R}^{d+1})} \leq C\left\| 1_{E} \right\|_{L^{p}(\mathbb{R}^{d+1})}. \nonumber
\end{equation}
\end{main}

The main difficulty in obtaining the endpoint restricted weak type estimate for the global $t$ arises from the weak type $\pose{1,1}$ estimate for each individual piece of the maximal operator $\MM$. More precisely, in \cite{ryu} and \cite{seeger}, the last two frequency variables in $\R^d\times\R$ are decomposed according to their size $2^k$, and a $k2^k$ bound is obtained for the weak type $\pose{1,1}$ estimate for each piece of the global maximal operator. Due to the additional $k$ loss in these weak type $\pose{1,1}$ estimates, the Bourgain interpolation trick cannot be applied to obtain the endpoint restricted weak type estimate. In contrast, in this paper we decompose the full frequency space $\R^d\times\R$ instead of only a partial set and employ a variant of the Calderón–Zygmund decomposition to eliminate these losses for the maximal averages. In particular, we use the parameter
$j$ instead of $k$. To achieve a $2^j$ bound for the weak type $\pose{1,1}$ estimate, we consider an application of the Calderón–Zygmund decomposition with respect to different heights and different maximal functions.

\paragraph{Organization of this paper}
In Section~2, we decompose the frequency space $\R^{d+1}$ into regions where $\llvert\xi\rrvert\sim2^j$ and $\llvert \xi_{d+1}\rrvert\sim2^{j+\ell}$. Moreover, we reduce the Main Theorem to obtaining a $2^j$ bound for the weak type $\pose{1,1}$ estimate and a $2^{-j\pose{\frac{d-2}{2}}}$ decay rate for the $L^2$ estimate via the Bourgain interpolation trick. Then, in Section~3, we introduce a new ball and a new scaled maximal operator \(\mathfrak{M}\) to construct a variant of the Calderón–Zygmund decomposition. Using this decomposition, we remove the worst part of the weak type $\pose{1,1}$ estimate—that is, we eliminate the portion of the domain that fails to satisfy the Hörmander type condition. Once this problematic part is removed, we demonstrate that the Hörmander type condition holds for the remaining portion, thereby establishing the desired $2^j$ bound for the weak type $\pose{1,1}$ estimate. For the $L^2$ estimate, we provide several reductions in Section~4, including the asymptotics of the Fourier transform of the spherical measure, and we obtain the $2^{-j\pose{\frac{d-2}{2}}}$ decay rate in Section~5 by employing an almost orthogonality argument. Finally, in the Appendix, we detail minor calculations concerning the small frequency term and the tail term in the  $L^2$ estimate, utilizing classical results for the Hardy–Littlewood maximal functions on the Heisenberg group.

\begin{notation}
For two scalars $\alpha$ and $\beta$, we write $\alpha\lesssim\beta$ if $\alpha\leq C\beta$ for a constant $C>0$ independent of $\alpha$ and $\beta$, allowing modifications of $C$ on a line-by-line basis. Additionally, we denote $\alpha\sim\beta$ if $\alpha\lesssim\beta$ and $\beta\lesssim\alpha$. We frequently use the notation $\lesssim$ in the sense that the implicit constant $C$ depends only the index $p$ and the dimension $d$. 

In this paper, we use the following smooth non-negative cutoff functions:
\begin{itemize}
\item $\psi\pose{\xi}$ is supported in $\left\{\xi\in\R^{d}: \llvert \xi\rrvert\leq 2\right\}$ and $\psi\pose{\xi}=1$ for $\llvert \xi\rrvert\leq 1$,
\item $\chi\pose{\xi}=\psi\pose{\xi}-\psi\pose{2\xi}$ is supported in $\left\{\xi \in \R^{d}: \frac{1}{2}\leq\llvert \xi\rrvert\leq 2\right\}$, with $d$ depending on the context,
\end{itemize}
allowing for slight line-by-line modifications of $\psi$ and $\chi$. Also, we use $\widetilde{\psi}$ and $\widetilde{\chi}$ as smooth functions similar to $\psi$ and $\chi$, allowing for slight line-by-line modifications.

To distinguish between the Euclidean convolution and the Heisenberg convolution in \eqref{H-convolution}, we use the following notation for the Euclidean convolution:
\begin{align}
f\ast g\pose{x,x_{d+1}}&=\int f\pose{x-y,x_{d+1}-y_{d+1}}g\pose{y,y_{d+1}}dydy_{d+1}.\label{defconvolution}
\end{align}
\end{notation}

\section{Frequency Decomposition and Bourgain Interpolation}
Utilizing the Fourier transform, we can express $\MA\pose{f}$ as follows:
$$\MA\pose{f}\pose{x,x_{d+1},t}=\int_{\R^{d+1}}e^{2\pi i(x\cdot\xi+x_{d+1}\xi_{d+1})}\widehat{\sm}\pose{t\left(\xi+\xi_{d+1}Jx\right)}\widehat{f}\pose{\xi,\xi_{d+1}}d\xi d\xi_{d+1}.$$
Decompose $\widehat{\sm}\pose{\xi}$ as
\[
\widehat{\sm}\pose{\xi}=\widehat{\sm}_{1}\pose{\xi}+\sum_{j=2}^{\infty}\widehat{\sm_{j}}\pose{\xi}\]
where 
\[
\widehat{\sm_{1}}\pose{\xi}=\widehat{\sm}\pose{\xi}\psi\pose{\frac{\xi}{2}}\ \text{and}\ \widehat{\sm_{j}}\pose{\xi}=\widehat{\sm}\pose{\xi}\chi\pose{\frac{\xi}{2^j}}\ \text{for}\ j\geq2.\] 
For $j\geq1$, we define the corresponding average and its maximal function by
\begin{align}
\MA_{j}\pose{f}\pose{\ux,t}&=f\ast_{\bH}\pose{\sm_{j}}_t\pose{\ux}\label{averageall}
\\&=\int_{\R^{d+1}}e^{2\pi i(x\cdot\xi+x_{d+1}\xi_{d+1})}\widehat{\sm_{j}}\pose{t\left(\xi+\xi_{d+1}Jx\right)}\widehat{f}\pose{\xi,\xi_{d+1}}d\xi d\xi_{d+1}\nonumber \ \text{and}
\\ \MM_{j}\pose{f}\pose{\ux}&=\sup_{t>0}\llvert\MA_{j}\pose{f}\pose{\ux,t}\rrvert.\nonumber
\end{align}
Then, we have  the pointwise inequality
\[
\MM\pose{f}\pose{\ux}\leq  \MM_{1}\pose{f}\pose{\ux}+\sum_{j=2}^{\infty}\MM_{j}\pose{f}\pose{\ux}.\]

We shall frequently use the following well-known estimates for $\sm_j$:
\begin{equation}\label{measureofsmj}
\llvert \sm_j\pose{x}\rrvert\leq 2^{jd}\int_{\R^d}\llvert\chi^{\vee}\pose{2^j \pose{x-u}}\rrvert\sm\pose{u}\leq \int_{\R^d}\frac{C_N2^{jd}}{\pose{1+2^j\llvert x-u\rrvert}^N}\sm\pose{u}\leq \frac{C_M2^j}{\pose{1+\llvert x\rrvert}^M}
\end{equation}
for every $N>M>d$. When we use \eqref{measureofsmj}, we frequently take $N, M>100d$. By using \eqref{measureofsmj} for the case of $j=1$, we show that $\llVert \MM_{1}\rrVert_{L^{p}\rightarrow L^{p}}\lesssim1$ for $1<p\leq\infty$ as detailed in the Appendix.

 To handle the second term (when $j\geq2$), we decompose $\sm_{j}$ into the inner term $\sm_{j}^{I}$ and the outer term $\sm_{j}^{O}$ defined by
\begin{align}
\sm_{j}^{I}\pose{x}&=\pose{2^{jd}\chi^{\vee}\pose{2^j \cdot}\psi\pose{\cdot}\ast\sm}\pose{x}=2^{jd}\int_{\R^d}\chi^{\vee}\pose{2^j \pose{x-u}}\psi\pose{x-u}\sm\pose{u},\label{insidesmj}
\\
\sm_{j}^{O}\pose{x}&=\pose{2^{jd}\chi^{\vee}\pose{2^j \cdot}\psi^{c}\pose{\cdot}\ast\sm}\pose{x}=2^{jd}\int_{\R^d}\chi^{\vee}\pose{2^j \pose{x-u}}\psi^{c}\pose{x-u}\sm\pose{u},\label{outsidesmj}
\end{align}
where $\psi^{c}\pose{x}=1-\psi\pose{x}$ is a smooth function supported in $\{x\in\R^d :\llvert x\rrvert\geq1\}$ and $*$ denotes the usual Euclidean convolution as in \eqref{defconvolution}. Similarly to $\MA_{j}$ and $\MM_{j}$, we define $\MA^{I}_{j}$, $\MM^{I}_{j}$, $\MA^{O}_{j}$, and $\MM^{O}_{j}$ by replacing $\widehat{\sm_{j}}$ with $\widehat{\sm^{I}_{j}}$ and $\widehat{\sm^{O}_{j}}$, respectively, in \eqref{averageall}. Then, we have
\[\sum_{j=2}^{\infty}\MM_{j}\pose{f}\pose{\ux}\leq\sum_{j=2}^{\infty}\MM^{I}_{j}\pose{f}\pose{\ux}+\sum_{j=2}^{\infty}\MM^{O}_{j}\pose{f}\pose{\ux}.\]

Due to \eqref{measureofsmj}, the rapidly decreasing property of $\chi^{\vee}$ and the support condition $\llvert x-u\rrvert\geq1$ in \eqref{outsidesmj}, we see that $\sm_{j}^{O}$ also satisfies the following rapidly decreasing estimate:
\begin{equation}
\llvert\sm_{j}^{O}\pose{x}\rrvert\leq C_{N}\frac{2^{-jN}}{\pose{1+\llvert x\rrvert}^{N}}\ \text{for any sufficiently large}\ N\in\N.\label{outdecreasing}
\end{equation}
Using \eqref{outdecreasing}, one may show that $\llVert\MM_{j}^{O}\rrVert_{L^{p}\rightarrow L^{p}}\lesssim2^{-jd}$ for $1<p\leq\infty$ by an argument analogous to that used for $\MM_{1}$ (see the Appendix).

From now on, it suffices to consider the sum of the inner terms $\sum_{j=2}^{\infty}\MM^{I}_{j}\pose{f}$. To decompose the $\xi_{d+1}$-variable, for $j\geq2$ and $\ell\in\N_{0}=\N\cup\{0\}$, define
\begin{align}\label{pieceofd+1}
\varphi_{j,\ell}\pose{\xi_{d+1}}&=\left\{\begin{array}{ll} \psi\pose{\frac{\xi_{d+1}}{2^j}}\text{if}\ \ell=0, \\  \chi\pose{\frac{\xi_{d+1}}{2^{j+\ell}}}\text{if}\ \ell\geq1,\end{array}\right. 
\end{align}
and set
\begin{align}\label{kernelofinnerterm}
\pose{K_{j,\ell}}_{t}\pose{x,x_{d+1}}&=\frac{1}{t^d}\sm_{j}^{I}\pose{\frac{x}{t}}\frac{1}{t^2}\varphi_{j,\ell}^{\vee}\pose{\frac{x_{d+1}}{t^2}}.
\end{align}
Then, we have $\MA^{I}_{j}\pose{f}=\sum_{\ell=0}^{\infty}\MA^{I}_{j,\ell}\pose{f}$ where
\begin{align}\label{inneraverage}
\MA^{I}_{j,\ell}\pose{f}\pose{\ux,t}&=f\ast_{\bH} \pose{K_{j,\ell}}_{t}\pose{\ux}
\\
&=\int_{\R^{d+1}} e^{2\pi i\ux\cdot\uxi}\widehat{\pose{K_{j,\ell}}_{t}}\pose{\pose{\xi+\xi_{d+1}Jx},\xi_{d+1}}\widehat{f}\pose{\xi,\xi_{d+1}}d\xi d\xi_{d+1}.\nonumber
\end{align} 
Taking the supremum over $t>0$, we define the maximal function corresponding to the inner term by
\begin{equation}\label{maximalfunctionofinnerterm}
\MM_{j,\ell}^{I}\pose{f}\pose{\ux}=\sup_{t>0}\llvert \MA_{j,\ell}^{I}\pose{f}\pose{\ux,t}\rrvert.
\end{equation}
It follows that
\[\sum_{j=2}^{\infty}\MM_{j}^{I}\pose{f}\leq\sum_{\ell=0}^{\infty}\sum_{j=2}^{\infty}\MM_{j,\ell}^{I}\pose{f},\]
and consequently,
\[\llVert \sum_{j=2}^{\infty}\MM_{j}^{I}\pose{f}\rrVert_{L^{p,\infty}\pose{\R^{d+1}}}\leq \llVert \sum_{\ell=0}^{\infty}\sum_{j=2}^{\infty}\MM_{j,\ell}^{I}\pose{f}\rrVert_{L^{p,\infty}\pose{\R^{d+1}}}.\]

Recall that 
\begin{equation}\label{definitionofLorentznorm}
\llVert f\rrVert_{L^{p,\infty}\pose{\R^{d+1}}}:=\sup_{\alpha>0}\left[\alpha\llvert \{\ux\in\R^{d+1}: \llvert f\pose{\ux}\rrvert>\alpha\}\rrvert^{\frac{1}{p}}\right].
\end{equation}
Therefore, to prove the Main Theorem, it suffices to show that 
\begin{equation}\label{targetpp}
\llvert \left\{\ux\in\R^{d+1} : \sum_{\ell=0}^{\infty}\sum_{j=2}^{\infty}\MM_{j,\ell}^{I}\pose{ 1_{E}}\pose{\ux}>\alpha\right\}\rrvert\lesssim\frac{1}{\alpha^p}\llvert E\rrvert
\end{equation}
for any finite measurable set $E\subseteq\R^{d+1}$ and for $p=\frac{d}{d-1}$.

\begin{ob}\label{proposition1}Let $1\leq p\leq2$. Assume that there exist constants $c, C>0$ such that
\begin{equation}\label{weakjsum}
\llVert \sum_{j=2}^{\infty}\MM^{I}_{j,\ell}\pose{ 1_{E}}\rrVert_{L^{p,\infty}\pose{\R^{d+1}}}\leq C2^{-c\ell}\llVert  1_{E}\rrVert_{L^p\pose{\R^{d+1}}}\ \text{for any finite measurable set}\ E.
\end{equation}
Then, \eqref{targetpp} holds for $1\leq p\leq2$.
\end{ob}
\begin{proof} Take $0<\epsilon <c$.
To prove \eqref{targetpp}, observe that
\[\left\{\ux\in\R^{d+1} : \sum_{\ell=0}^{\infty}\sum_{j=2}^{\infty}\MM_{j,\ell}^{I}\pose{ 1_{E}}\pose{\ux}>\alpha\right\}\subseteq \bigcup_{l=0}^{\infty}\left\{\ux\in\R^{d+1}: \sum_{j=2}^{\infty}\MM_{j,\ell}^{I}\pose{ 1_{E}}\pose{\ux}>\pose{1-2^{-\epsilon}}\frac{\alpha}{2^{\epsilon\ell}}\right\}.\]
Then, by \eqref{definitionofLorentznorm} and \eqref{weakjsum}, we have
\begin{align*}
\llvert\left\{\ux\in\R^{d+1}: \sum_{j=2}^{\infty}\MM_{j,\ell}^{I}\pose{ 1_{E}}\pose{\ux}>\pose{1-2^{-\epsilon}}\frac{\alpha}{2^{\epsilon\ell}}\right\}\rrvert\leq\frac{C^p2^{-pc\ell}}{\pose{1-2^{-\epsilon}}^p 2^{-p\epsilon \ell}\alpha^p}\llvert E\rrvert =\frac{C 2^{-\pose{c-\epsilon}p\ell}}{\alpha^p}\llvert E\rrvert.
\end{align*}
Since $\epsilon <c$, summing over \(\ell\) yields
\begin{align*}
\llvert\left\{\ux\in\R^{d+1} : \sum_{\ell=0}^{\infty}\sum_{j=2}^{\infty}\MM_{j,\ell}^{I}\pose{ 1_{E}}\pose{\ux}>\alpha\right\}\rrvert\leq\sum_{\ell=0}^{\infty}\frac{C 2^{-\pose{c-\epsilon }p\ell}}{\alpha^p}\llvert E\rrvert\leq\frac{C }{\alpha^p}\llvert E\rrvert,
\end{align*}
which is \eqref{targetpp}.
\end{proof}

By Observation \ref{proposition1}, it suffices to prove that $\sum_{j=2}^{\infty}\MM_{j,\ell}^{I}$ is of restricted weak type $\pose{p,p}$ for $p=\frac{d}{d-1}$ with a $2^{-\ell}$ bound:
\begin{equation}\label{weakjsum2}
\llVert \sum_{j=2}^{\infty}\MM^{I}_{j,\ell}\pose{ 1_{E}}\rrVert_{L^{p,\infty}\pose{\R^{d+1}}}\lesssim2^{-\ell}\llVert  1_{E}\rrVert_{L^p\pose{\R^{d+1}}}\ \text{for any finite measurable set}\ E.
\end{equation}
To show \eqref{weakjsum2}, we utilize the Bourgain interpolation trick in (\cite{BI}, \cite{Carberyinterpolation}) to sum over $j$. More precisely, suppose that for a family of sublinear operators $T_j$ with $j\geq1$, we have
\begin{align*}
\llVert T_{j}\rrVert_{L^{p_{0},1}\rightarrow L^{q_{0},\infty}}&\lesssim C2^{ja_{0}},\\
\llVert T_{j}\rrVert_{L^{p_{1},1}\rightarrow L^{q_{1},\infty}}&\lesssim D2^{-ja_{1}}
\end{align*}
for some $p_{0},q_{0},p_{1},q_{1}\in[1,\infty]$ and $a_{0},a_{1}>0$. Then, the sum $\sum_{j=1}^{\infty}T_{j}$ is of restricted weak type $(p,q)$ with a $C^{1-\theta}D^{\theta}$ bound where 
\[\pose{\frac{1}{p},\frac{1}{q},0}=\pose{1-\theta}\pose{\frac{1}{p_{0}},\frac{1}{q_{0}},a_{0}}+\theta\pose{\frac{1}{p_{1}},\frac{1}{q_{1}},-a_{1}}\ \text{with}\ \theta=\frac{a_0}{a_0+a_1}\in\pose{0,1}.\]

\begin{ob} The following two inequalities imply \eqref{weakjsum2}:
\begin{align}
\llVert \MM_{j,\ell}^{I}\rrVert_{L^1\rightarrow L^{1,\infty}}&\leq C 2^j,\label{weak11}
\\
\llVert \MM_{j,\ell}^{I}\rrVert_{L^2\rightarrow L^2}&\leq C 2^{-j\pose{\frac{d-2}{2}}}2^{-\ell\pose{\frac{d}{2}}}.\label{l2estimate}
\end{align}
\end{ob}
\begin{proof} For a fixed $\ell\geq0$, set $T_j=\MM_{j,\ell}^{I}$. To apply the above interpolation argument, choose 
\[p_0=q_0=1,\quad p_1=q_1=2,\quad p=q=\frac{d}{d-1},\quad \text{and}\quad \theta=\frac{2}{d}.
\]
Then the estimates \eqref{weak11} and \eqref{l2estimate} yield the desired restricted weak type \((p,p)\) bound with a \(2^{-\ell}\) factor, which is precisely \eqref{weakjsum2}.
\end{proof}
We shall prove \eqref{weak11} in Section \ref{sectionweaktype11} and \eqref{l2estimate} in Section \ref{sectionreductionl2} and \ref{sectionl2estimatesformainterms}.

\section{Weak Type $\pose{1,1}$ Estimate for $\MM_{j,\ell}^{I}$}\label{sectionweaktype11}

In this section, we shall prove the inequality \eqref{weak11}:

\begin{theorem}\label{weak11theorem}
For any fixed $j\geq2$ and $\ell\geq0$, the maximal operator $\MM_{j,\ell}^{I}$ is of weak type $\pose{1,1}$ with a $2^j$ bound. That is, for any $f\in L^1\pose{\R^{d+1}}$ and $\alpha>0$, it holds that
\begin{equation}\label{weakmainestimate}
\llvert \left\{\ux\in\R^{d+1}: \MM_{j,\ell}^{I}\pose{f}\pose{\ux}>\alpha \right\}\rrvert\leq\frac{C 2^j}{\alpha}\llVert f\rrVert_{L^1\pose{\R^{d+1}}}
\end{equation}
where the positive constant $C$ is independent of $j$ and $\ell$.
\end{theorem}

\subsection{Preliminaries}

Fix $j\geq2$ and $\ell\geq0$. To establish \eqref{weakmainestimate}, we define a new ball $B\pose{\ux,r}$ for $r>0$ by
\begin{equation}\label{newball}
B\pose{\ux,r}=\left\{ \uy=\pose{y,y_{d+1}}\in\R^{d+1}: \llvert x-y\rrvert<\frac{r}{2^{j+\ell}}\ \text{and} \ \llvert x_{d+1}-y_{d+1}+\langle Jx,y\rangle\rrvert<\frac{r^2}{2^{j+\ell}}\right\}.
\end{equation}
We shall show that $B\pose{\ux,r}$ satisfies the doubling property and that the associated Hardy-Littlewood maximal function is of weak type $\pose{1,1}$.

\begin{lemma}[Doubling property]\label{doubling} Let $0<s\leq r$.
\item[(1)] For any $\ux, \uy\in\R^{d+1}$, $B\pose{\ux,r}\cap B\pose{\uy,s}\neq \emptyset$ implies that $B\pose{\uy,s}\subseteq B\pose{\ux,3r}$,
\item[(2)] $\llvert B\pose{\ux,3r}\rrvert=3^{d+2}\llvert B\pose{\ux,r}\rrvert$.
\end{lemma}
\begin{proof} (1): Let $\uz\in B\pose{\ux,r}\cap B\pose{\uy,s}$. Then, we have the following inequalities:
\begin{align*}
\llvert x-z\rrvert,\ \llvert y-z\rrvert&<\frac{r}{2^{j+\ell}}, \\
\llvert x_{d+1}-z_{d+1}+\langle Jx,z\rangle\rrvert, \ \llvert y_{d+1}-z_{d+1}+\langle Jy,z\rangle\rrvert &< \frac{r^2}{2^{j+\ell}}.
\end{align*}
Given $\bar{u}\in B\pose{\uy,s}$, we have 
\begin{align*}
\llvert u-x\rrvert&\leq \llvert u-y\rrvert+\llvert y-z\rrvert+\llvert z-x\rrvert<\frac{3r}{2^{j+\ell}}, \\
\llvert x_{d+1}-u_{d+1}+\langle Jx,u\rangle\rrvert&\leq\llvert x_{d+1}-z_{d+1}+\langle Jx,z\rangle\rrvert+\llvert z_{d+1}-y_{d+1}-\langle Jy,z\rangle\rrvert
\\&+\llvert y_{d+1}-u_{d+1}+\langle Jy,u\rangle\rrvert+\llvert \langle J\pose{x-y},u-z\rangle\rrvert\\
&\leq \frac{3r^2}{2^{j+\ell}}+\llvert \langle J\pose{x-y},u-z\rangle\rrvert.
\end{align*}
Note that $\llvert \langle J\pose{x-y},u-z\rangle\rrvert\leq \llvert x-y\rrvert \llvert u-z\rrvert\leq \frac{2r}{2^{j+\ell}}\times\frac{2s}{2^{j+\ell}}\leq\frac{4r^2}{2^{2\pose{j+\ell}}}<\frac{4r^2}{2^{j+\ell}}$ from $j, \ell\geq0$. Thus $\llvert x_{d+1}-u_{d+1}+\langle Jx,u\rangle\rrvert<\frac{9r^2}{2^{j+\ell}}$ and so $\bar{u}\in B\pose{\ux,3r}$.

(2): Consider the Euclidean measure of $B\pose{\ux,3r}$ for an arbitrary $\ux\in\R^{d+1}$.
\begin{align*}
\llvert B\pose{\ux,3r}\rrvert&=\int_{\llvert x-y\rrvert<\frac{3r}{2^{j+\ell}}}\int_{\llvert x_{d+1}-y_{d+1}+\langle Jx,y\rangle\rrvert<\frac{9r^2}{2^{j+\ell}}}dy_{d+1}dy \\
&=\int_{\llvert y\rrvert<\frac{3r}{2^{j+\ell}}}\int_{\llvert y_{d+1}\rrvert<\frac{9r^2}{2^{j+\ell}}}dy_{d+1}dy \\
&=3^{d+2}\int_{\llvert x-y\rrvert<\frac{r}{2^{j+\ell}}}\int_{\llvert x_{d+1}-y_{d+1}+\langle Jx,y\rangle\rrvert<\frac{r^2}{2^{j+\ell}}}dy_{d+1}dy=3^{d+2}\llvert B\pose{\ux,r}\rrvert.
\end{align*}
Therefore, the new ball $B\pose{\ux,r}$ in \eqref{newball} has the doubling property.
\end{proof}
In fact, our new ball $B\pose{\ux,r}$ is thicker than the original Carnot–Carathéodory ball with dimensions $\pose{\frac{r}{2^{j+\ell}}}^d\times \pose{\frac{r}{2^{j+\ell}}}^2$, thereby satisfying the doubling property. Now, using standard techniques associated with the doubling property, we can construct both the Whitney-type decomposition and the Calderón-Zygmund decomposition required for our analysis. As a preliminary step, we establish a Vitali covering lemma for coverings composed of $B\pose{\ux,r}$.

\begin{lemma}[Vitali Covering Lemma]\label{Vitali1}
Let $S$ be a measurable set in $\R^{d+1}$ with $S\subseteq \bigcup_{i=1}^{n}B\pose{\ux_i,r_i}$ for some $\ux_i\in\R^{d+1}$ and $r_i >0$. Then there exists a disjoint subcollection $\{ B_{i}\}_{i=1}^{m}$ of $\{ B\pose{\ux_i,r_i}\}_{i=1}^{n}$ satisfying
\[\sum_{i=1}^m \llvert B_i \rrvert \geq 3^{-\pose{d+2}}\llvert S\rrvert. \]
\end{lemma}
\begin{proof}
See the proof of \cite[Lemma~1,p.~12]{steinhm}. In that proof, only the doubling property of the balls $B(\ux,r)$ is required. By combining this fact with our Lemma~\ref{doubling}, we can follow the same argument to obtain the desired conclusion.
\end{proof}

Now, we define a maximal function $\mathfrak{M}_{j,\ell}\pose{f}$ by
\begin{equation}\label{ourhardy}
\mathfrak{M}_{j,\ell}\pose{f}\pose{\ux}=\sup_{\substack{r>0, \\ \ux\in B\pose{\uc,r},\\ \uc\in\R^{d+1}}}\frac{1}{\llvert B\pose{\uc,t}\rrvert}\int_{B\pose{\uc,r}}\llvert f\pose{\uy}\rrvert d\uy.
\end{equation}

\begin{lemma}\label{weak11R}
The maximal function $\mathfrak{M}_{j,\ell}$ defined above is of weak type $(1,1)$. That is, there exists a constant $C > 0$, independent of $j$ and $\ell$, such that for all $\alpha > 0$ and for all $f \in L^1(\mathbb{R}^{d+1})$,
\begin{equation}\label{weak11Rhardy}
\llvert\left\{\ux \in \mathbb{R}^{d+1} : \mathfrak{M}_{j,\ell}\pose{f}\pose{\ux} > \alpha\right\}\rrvert\leq\frac{C}{\alpha}\llVert f\rrVert_{L^1\pose{\mathbb{R}^{d+1}}}.
\end{equation}
\end{lemma}

\begin{proof}
Arguing similarly as in Lemma~\ref{Vitali1}, we see that the claim follows directly from the proof of \cite[Theorem~1, p.~13]{steinhm}.
\end{proof}

\subsection{Proof of Theorem \ref{weak11theorem}}

In this subsection, we prove Theorem~\ref{weak11theorem} by focusing on the two sets:
\begin{align*}
F_{\alpha} &= \left\{ \ux \in \R^{d+1} : \MM_{j,\ell}^{I}(f)(\ux) > \alpha \right\}, \\
E_{\alpha} &= \left\{ \ux \in \R^{d+1} : \mathfrak{M}_{j,\ell}(f)(\ux) > \alpha \right\},
\end{align*}
where the target set \( F_{\alpha} \) corresponds to the left-hand side of \eqref{weakmainestimate} and the exceptional set \( E_{\alpha} \) corresponds to the left-hand side of \eqref{weak11Rhardy}.

In the traditional approach to proving the weak type \((1,1)\) boundedness of a maximal function via the Calderón–Zygmund decomposition, the target maximal function in $F_{\alpha}$ and the maximal function in $E_\alpha$ used in the decomposition (which defines the exceptional set $E_{\alpha}$) share the same form. Also, the corresponding target set $F_{\alpha}$ and the exceptional set $E_\alpha$ share the same height $\alpha$. For further details on more advanced applications of the Calderón–Zygmund decomposition, see, e.g., \cite{grafakos} and \cite{TaoSeegerJames}. However, these approaches are not suitable for establishing the weak type $(1,1)$ boundedness of our maximal function $\mathcal{M}_{j,\ell}^I$.

To prove Theorem~\ref{weak11theorem}, we develop a variant of the Calderón-Zygmund decomposition for the exceptional set $E_{\lambda(\alpha)}$, where we employ a different height \(\lambda(\alpha) \neq \alpha\) and the different maximal function \(\mathfrak{M}_{j,\ell} \neq \MM_{j,\ell}^{I}\) as defined in \eqref{ourhardy}.

\begin{proposition}[Calde\'{o}n-Zygmund Decomposition]\label{Calderon}
Given $f\in L^1\pose{\R^{d+1}}$ and $\lambda\pose{\alpha}>0$, let $E_{\lambda\pose{\alpha}}=\left\{\ux\in\R^{d+1} : \mathfrak{M}_{j,\ell}\pose{f}\pose{\ux}>\lambda\pose{\alpha}\right\}$. Then there exist disjoint collections $\{B_k\}_{k=1}^{\infty}$ and $\{Q_k \}_{k=1}^{\infty}$ where $B_k=B\pose{\uc_k,r_k}$ is defined in \eqref{newball} such that
\begin{equation}\label{relation1111}
B_k\subseteq Q_k\subseteq B^*_k:=B\pose{\uc_{k},c_1r_k}\ \text{and}\ E_{\lambda\pose{\alpha}}=\bigsqcup_{k=1}^{\infty}Q_{k}=\bigcup_{k=1}^{\infty}B^*_k
\end{equation}
where we set $c_1=6^2$ with respect to the doubling constant $3$ in Lemma \ref{doubling}. With the above decomposition for $E_{\lambda\pose{\alpha}}$, there exists also a decomposition of $f=g+\sum_{k=1}^{\infty}b_k$ so that
\begin{itemize}
\item[(1)] $g=f1_{E_{\lambda\pose{\alpha}}^c}+\sum_{k=1}^{\infty}f_{Q_k}1_{Q_k}$ satisfies $\llvert g\pose{\ux}\rrvert\leq C \lambda\pose{\alpha}$ for almost everywhere $\ux\in\R^{d+1}$ where $f_{Q_k}=\frac{1}{\llvert Q_k\rrvert}\int_{Q_k} f\pose{\uy}d\uy$,
\item[(2)] each $b_k=\pose{f-f_{Q_k}}1_{Q_k}$ is supported in $Q_k\subseteq B^*_k$,
\[\int \llvert b_k\pose{\ux}\rrvert d\ux\leq C \lambda\llvert B^*_k\rrvert\ \text{and}\ \int b_{k}\pose{\ux}d\ux=0,\] 
\item[(3)] $\sum_{k=1}^{\infty}\llvert B^{*}_k\rrvert\leq\frac{C }{\lambda\pose{\alpha}}\llVert f\rrVert_{L^1\pose{\R^{d+1}}}$.
\end{itemize}
\end{proposition}
\begin{proof} See the proof of \cite[Lemma~2, p.~15]{steinhm} and \cite[Theorem~2, p.~17]{steinhm} with $\lambda=\lambda\pose{\alpha}$. 
\end{proof}

 Now, we shall prove Theorem~\ref{weak11theorem}. To establish \eqref{weakmainestimate}, we employ the Calderón–Zygmund decomposition introduced above with 
\begin{equation}
\lambda\pose{\alpha}=2^{j\pose{d-1}}2^{\ell d}\alpha.
\end{equation}
Then, our exceptional set \(E_{\lambda(\alpha)}\) is given by
\begin{align*}
E_{\lambda\pose{\alpha}}&=\left\{\ux\in\R^{d+1}: M_{R}\pose{f}\pose{\ux}>\lambda\pose{\alpha}\right\}
\\
&=\left\{\ux\in\R^{d+1}: M_{R}\pose{f}\pose{\ux}>2^{j\pose{d-1}}2^{\ell d}\alpha\right\}.
\end{align*}
The set 
\[E_{\lambda\pose{\alpha}}=E_{2^{j\pose{d-1}}2^{\ell d}\alpha}\]
is the key component of this section. We can replace $\alpha$ in $E_\alpha$ by $2^{j\pose{d-1}}2^{\ell d}\alpha$ thanks to the good $L^2$-decay rate 
\[2^{-j\pose{\frac{d-2}{2}}}2^{-\ell\pose{\frac{d}{2}}}\]
of $\MM^{I}_{j,\ell}$ in \eqref{l2estimate}. With this choice of height $\lambda\pose{\alpha}=2^{j\pose{d-1}}2^{\ell d}\alpha$, we first control a good portion of $F_\alpha$ as indicated in \eqref{goodpart} below. On the other hand, note that this new exceptional set $E_{\lambda\pose{\alpha}}$ is smaller than the traditional exceptional set $E_{\alpha}$ since $\lambda\pose{\alpha}=2^{j\pose{d-1}}2^{\ell d}\alpha>\alpha$. Hence, we can enlarge $E_{\lambda\pose{\alpha}}$ to $\widetilde{E}_{\lambda\pose{\alpha}}$ as needed. Using this enlarged set $\widetilde{E}_{\lambda\pose{\alpha}}$, we remove the worst part in the remaining bad portion of $F_\alpha$, where the H\"ormander type condition \eqref{Hormander110} below cannot be controlled. After this elimination, we conclude \eqref{weakmainestimate} by verifying the Hörmander type condition \eqref{Hormander110} in Subsection~\ref{subsectionhormander}.

 Let $f=g+\sum_{k=1}^{\infty} b_k$ be the Calder\'{o}n-Zygmund decomposition of $f$ with $E_{\lambda\pose{\alpha}}=E_{2^{j\pose{d-1}}2^{\ell d}\alpha}$ as defined in Proposition \ref{Calderon}. Then, to prove \eqref{weakmainestimate}, it suffices to show the following two inequalities: 
\begin{align}
\llvert\left\{\ux\in\R^{d+1}: \MM^{I}_{j,\ell}\pose{g}\pose{\ux}>\frac{\alpha}{2}\right\}\rrvert&\leq\frac{C 2^j}{\alpha}\llVert f\rrVert_{L^1\pose{\R^{d+1}}},\label{goodpart}
\\
\llvert\left\{\ux\in\R^{d+1}: \MM^{I}_{j,\ell}\pose{\sum_{k=1}^{\infty}b_k}\pose{\ux}>\frac{\alpha}{2}\right\}\rrvert&\leq\frac{C 2^j}{\alpha}\llVert f\rrVert_{L^1\pose{\R^{d+1}}}\label{badpart}.
\end{align}
\begin{proof}[\textbf{Proof of \eqref{goodpart}}] By applying the Chebyshev inequality to the left-hand side of \eqref{goodpart} and using the $L^2$ decay $2^{-j\pose{\frac{d-2}{2}}}2^{-\ell\pose{\frac{d}{2}}}$ of $\MM^{I}_{j,\ell}$ in \eqref{l2estimate}, we obtain
\begin{align}\label{3.5eq}
\llvert\left\{\ux\in\R^{d+1}: \MM^{I}_{j,\ell}\pose{g}\pose{\ux}>\frac{\alpha}{2}\right\}\rrvert&\leq\frac{4}{\alpha^2}\llVert \MM^{I}_{j,\ell}\pose{g}\rrVert^2_{L^2\pose{\R^{d+1}}}\leq\frac{C 2^{-j\pose{d-2}}2^{-\ell d}}{\alpha^2}\int\llvert g\pose{\ux}\rrvert^2d\ux.
\end{align}
From (1) of Proposition \ref{Calderon} with $\lambda\pose{\alpha}=2^{j\pose{d-1}}2^{\ell d}\alpha$, we have 
\begin{align*}
\llvert g\pose{\ux}\rrvert\leq C 2^{j\pose{d-1}}2^{\ell d}\alpha\ \text{for almost everywhere}\ \ux\in\R^{d+1}.
\end{align*}
Substituting this into the right-hand side of \eqref{3.5eq}, we get
\begin{align*}
\llvert\left\{\ux\in\R^{d+1}: \MM^{I}_{j,\ell}\pose{g}\pose{\ux}>\alpha\right\}\rrvert\leq\frac{C 2^{j}}{\alpha}\int\llvert g\pose{\ux}\rrvert d\ux
\leq\frac{C 2^j}{\alpha}\llVert f\rrVert_{L^1\pose{\R^{d+1}}},
\end{align*}
which implies \eqref{goodpart}.
\end{proof}
\begin{proof}[\textbf{Proof of \eqref{badpart}}]
Now, we turn our attention to proving \eqref{badpart}. First, we enlarge the exceptional set 
\[E_{\lambda\pose{\alpha}}=E_{2^{j\pose{d-1}}2^{\ell d}\alpha}
\] by anisotropic dilation. Recall the definition of our new ball $B\pose{\uc,r}$ in \eqref{newball}:
\begin{align*}
B\pose{\uc,r}=\left\{ \uy=\pose{y,y_{d+1}}\in\R^{d+1}: \llvert c-y\rrvert<\frac{r}{2^{j+\ell}}\ \text{and} \ \llvert c_{d+1}-y_{d+1}+\langle Jc,y\rangle\rrvert<\frac{r^2}{2^{j+\ell}}\right\}.
\end{align*}
We then define the enlarged ball $\widetilde{B}\pose{\uc,r}$ by
\begin{align}\label{enlargedball}
\widetilde{B}\pose{\uc,r}&=\left\{ \uy=\pose{y,y_{d+1}}\in\R^{d+1}: \llvert c-y\rrvert<30r\ \text{and} \ \llvert c_{d+1}-y_{d+1}+\langle Jc,y\rangle\rrvert<\frac{(30r)^2}{2^{j+\ell}}\right\}.
\end{align}
Accordingly, we enlarge $E_{\lambda\pose{\alpha}}=\bigcup_{k=1}^{\infty}B^*_k$ (as in \eqref{relation1111}) to 
\begin{align}\label{enlargedexceptionalset}
\widetilde{E}_{\lambda\pose{\alpha}}=\bigcup_{k=1}^{\infty}\widetilde{B^*_k}=\bigcup_{k=1}^{\infty}\widetilde{B}\pose{\bar{c}_{k},c_1r_k}.
\end{align}

To prove \eqref{badpart}, we decompose the set $\left\{\ux\in\R^{d+1}: \MM^{I}_{j,\ell}\pose{\sum_{k=1}^{\infty}b_k}\pose{\ux}>\frac{\alpha}{2}\right\}$ as follows:
\begin{align*}
\llvert\left\{\ux\in\R^{d+1}: \MM^{I}_{j,\ell}\pose{\sum_{k=1}^{\infty}b_k}\pose{\ux}>\frac{\alpha}{2}\right\}\rrvert&\leq
\llvert\left\{\ux\in\widetilde{E}_{\lambda\pose{\alpha}}: \MM^{I}_{j,\ell}\pose{\sum_{k=1}^{\infty}b_k}\pose{\ux}>\frac{\alpha}{4}\right\}\rrvert\\
&+\llvert\left\{\ux\in\pose{\widetilde{E}_{\lambda\pose{\alpha}}}^c: \MM^{I}_{j,\ell}\pose{\sum_{k=1}^{\infty}b_k}\pose{\ux}>\frac{\alpha}{4}\right\}\rrvert.
\end{align*}
From \eqref{enlargedexceptionalset}, we have
\begin{align*}
\llvert\left\{\ux\in\widetilde{E}_{\lambda\pose{\alpha}}: \MM^{I}_{j,\ell}\pose{\sum_{k=1}^{\infty}b_k}\pose{\ux}>\frac{\alpha}{4}\right\}\rrvert\leq\llvert\widetilde{E}_{\lambda\pose{\alpha}}\rrvert=\llvert\bigcup_{k=1}^{\infty}\widetilde{B^*_k}\rrvert&\leq C 2^{\pose{j+\ell}d}\llvert\bigcup_{k=1}^{\infty}B^*_k\rrvert
\\&=C 2^{\pose{j+\ell}d}\llvert E_{\lambda\pose{\alpha}}\rrvert.
\end{align*}
By Lemma \ref{weak11R} with $\lambda\pose{\alpha}=2^{j\pose{d-1}}2^{\ell d}\alpha$, we obtain  
\begin{align*}
\llvert E_{\lambda\pose{\alpha}}\rrvert\leq\frac{C }{\lambda\pose{\alpha}}\llVert f\rrVert_{L^1\pose{\R^{d+1}}}=\frac{C }{2^{j\pose{d-1}}2^{\ell d}\alpha}\llVert f\rrVert_{L^1\pose{\R^{d+1}}}.
\end{align*}
Therefore, it holds that
\begin{align*}
\llvert\left\{\ux\in\widetilde{E}_{\lambda\pose{\alpha}}: \MM^{I}_{j,\ell}\pose{\sum_{k=1}^{\infty}b_k}\pose{\ux}>\frac{\alpha}{4}\right\}\rrvert\leq \frac{C 2^{\pose{j+\ell}d}}{2^{j\pose{d-1}}2^{\ell d}\alpha}\llVert f\rrVert_{L^1\pose{\R^{d+1}}}=\frac{C 2^j}{\alpha}\llVert f\rrVert_{L^1\pose{\R^{d+1}}}.
\end{align*}
Thus, to show \eqref{badpart}, it suffices to show that
\begin{align*}
\llvert\left\{\ux\in\pose{\widetilde{E}_{\lambda\pose{\alpha}}}^c: \MM^{I}_{j,\ell}\pose{\sum_{k=1}^{\infty}b_k}\pose{\ux}>\frac{\alpha}{4}\right\}\rrvert\leq\frac{C 2^j}{\alpha}\llVert f\rrVert_{L^1\pose{\R^{d+1}}}.
\end{align*}

Rather than proving the above inequality, we shall establish the following estimate:
\begin{equation}
\llVert \MM^{I}_{j,\ell}\pose{\sum_{k=1}^{\infty}b_k} \rrVert_{L^1\pose{\pose{\widetilde{E}_{\lambda\pose{\alpha}}}^c}}\leq C 2^j\llVert f\rrVert_{L^1\pose{\R^{d+1}}}.
\end{equation}
Recall the definitions of our spherical maximal operator $\MM^{I}_{j,\ell}$ in \eqref{maximalfunctionofinnerterm} and its kernel in \eqref{kernelofinnerterm}:
\begin{align*}
\MM^{I}_{j,\ell}\pose{f}\pose{\ux}&=\sup_{t>0}\llvert f\ast_{\bH}\pose{K_{j,\ell}}_{t}\pose{\ux}\rrvert \ \text{where}
\\
\pose{K_{j,\ell}}_{t}\pose{x,x_{d+1}}&=\left\{\begin{array}{ll}\frac{1}{t^d}\sm^{I}_{j}\pose{\frac{x}{t}}\frac{2^{j}}{t^2}\psi^{\vee}\pose{\frac{x_{d+1}}{t^2 /2^{j}}}\ \text{if}\ \ell=0, \\\frac{1}{t^d}\sm^{I}_{j}\pose{\frac{x}{t}}\frac{2^{j+\ell}}{t^2}\chi^{\vee}\pose{\frac{x_{d+1}}{t^2 /2^{j+\ell}}}\ \text{if}\ \ell\geq1. \end{array}\right. 
\end{align*}
Since $\int b_{k}\pose{\ux}d\ux=0$ (see (2) in Proposition \ref{Calderon}), we have
\begin{align*}
b_k\ast_{\bH}\pose{K_{j,\ell}}_{t}\pose{\ux}&=\int_{\R^{d+1}}b_{k}\pose{\pose{\uy}^{-1}\cdot\ux}\pose{K_{j,\ell}}_{t}\pose{\uy}d\uy=\int_{Q_k}\pose{K_{j,\ell}}_{t}\pose{\ux\cdot\pose{\uy}^{-1}}b_{k}\pose{\uy}d\uy \\
&=\int_{Q_k}\left[\pose{K_{j,\ell}}_{t}\pose{\ux\cdot\pose{\uy}^{-1}}-\pose{K_{j,\ell}}_{t}\pose{\ux\cdot\pose{\bar{c}_k}^{-1}}\right]b_{k}\pose{\uy}d\uy
\end{align*}
where $\bar{c}_k$ is the center of $B_k=B\pose{\bar{c}_k,r_k}\subseteq Q_k\subseteq B^*_k$. Thus, it holds that
\begin{align*}
\llVert \MM^{I}_{j,\ell}\pose{b_k} \rrVert_{L^1\pose{\pose{\widetilde{E}_{\lambda\pose{\alpha}}}^c}}&\leq\int_{Q_k}\left(\int_{\pose{\widetilde{B^*_k}}^c}\sup_{t>0}\llvert\pose{K_{j,\ell}}_{t}\pose{\ux\cdot\pose{\uy}^{-1}}-\pose{K_{j,\ell}}_{t}\pose{\ux\cdot\pose{\bar{c}_k}^{-1}}\rrvert d\ux\right)\\
&\times\llvert b_{k}\pose{\uy}\rrvert d\uy. 
\end{align*}

Now, if we can show that the following H\"ormander type condition
\begin{equation}\label{Hormander11}
\sup_{r>0}\sup_{\uy\in B\pose{\bar{c}_k,r}}\int_{\pose{\widetilde{B}\pose{\bar{c}_{k},r}}^c}\sup_{t>0}\llvert\pose{K_{j,\ell}}_{t}\pose{\ux\cdot\pose{\uy}^{-1}}-\pose{K_{j,\ell}}_{t}\pose{\ux\cdot\pose{\bar{c}_k}^{-1}}\rrvert d\ux\leq C 2^j,
\end{equation}
holds, then we obtain
\begin{align*}
\llVert \MM^{I}_{j,\ell}\pose{\sum_{k=1}^{\infty}b_k} \rrVert_{L^1\pose{\pose{\widetilde{E}_{\lambda\pose{\alpha}}}^c}}&\leq C  2^j\sum_{k=1}^{\infty}\int_{Q_k}\llvert b_{k}\pose{\uy}\rrvert d\uy \\
&\leq C 2^j\lambda\pose{\alpha}\sum_{k=1}^{\infty}\llvert B^*_k\rrvert \\
&\leq C 2^j\lambda\pose{\alpha}\frac{1}{\lambda\pose{\alpha}}\llVert f\rrVert_{L^1\pose{\R^{d+1}}}\\
&= C 2^j\llVert f\rrVert_{L^1\pose{\R^{d+1}}}.
\end{align*}
Thus, \eqref{badpart} follows once we establish \eqref{Hormander11}. We complete the proof of Theorem 3.1 by proving \eqref{Hormander11} in the following subsection.
\end{proof}

\subsection{H\"ormander Type Condition}\label{subsectionhormander}

Rather than proving \eqref{Hormander11} directly, by translating the center $\uc_k$ of $B_k$ and $\widetilde{B^*_k}$ to $0$, we shall prove the following H\"ormander type condition:
\begin{equation}\label{Hormander110}
\sup_{r>0}\sup_{\uy\in B\pose{0,r}}\int_{\pose{\widetilde{B}\pose{0,r}}^c}\sup_{t>0}\llvert\pose{K_{j,\ell}}_{t}\pose{\ux\cdot\pose{\uy}^{-1}}-\pose{K_{j,\ell}}_{t}\pose{\ux}\rrvert d\ux\leq C 2^j.
\end{equation}
To show \eqref{Hormander110}, we consider the regions of $\ux$ and $\uy$ in \eqref{Hormander110}. Put $\uc=0$ to \eqref{newball} and \eqref{enlargedball}. Then we have
\begin{align}
B\pose{0,r}&=\left\{ \uy=\pose{y,y_{d+1}}\in\R^{d+1}: \llvert y\rrvert<\frac{r}{2^{j+\ell}}\ \text{and} \ \llvert y_{d+1}\rrvert<\frac{r^2}{2^{j+\ell}}\right\},\label{domainofy}
\\
\widetilde{B}\pose{0,r}^c&=\left\{\ux=\pose{x,x_{d+1}}\in\R^{d+1}: \llvert x\rrvert>30r\ \text{or} \ \llvert x_{d+1}\rrvert>\frac{(30r)^2}{2^{j+\ell}}\right\}.\label{domainofx}
\end{align}
From now on, we fix $r>0$, $\llvert y\rrvert<\frac{r}{2^{j+\ell}}$, and $\llvert y_{d+1}\rrvert<\frac{r^2}{2^{j+\ell}}$.

In this subsection, we do not distinguish between $\psi$ and $\chi$ and hence use a smooth compact supported function $\psi$ for all cases $\ell\geq0$. Recall the definition of $\pose{K_{j,\ell}}_{t}$ in \eqref{kernelofinnerterm}:
\begin{align*}
\pose{K_{j,\ell}}_{t}\pose{x,x_{d+1}}&=\frac{1}{t^d}\sm_{j}^{I}\pose{\frac{x}{t}}\frac{2^{j+\ell}}{t^2}\psi^{\vee}\pose{\frac{x_{d+1}}{t^2 /2^{j+\ell}}}.
\end{align*}
Then, it holds that
\begin{align*}
\llvert\pose{K_{j,\ell}}_{t}\pose{\ux\cdot\pose{\uy}^{-1}}-\pose{K_{j,\ell}}_{t}\pose{\ux}\rrvert\leq\llvert E\pose{\ux,\uy,t}\rrvert+\llvert H\pose{\ux,\uy,t}\rrvert
\end{align*}
where the Euclidean difference part $E\pose{\ux,\uy,t}$ and the Heisenberg difference part $H\pose{\ux,\uy,t}$ are defined by
\begin{align}
E\pose{\ux,\uy,t}=\frac{1}{t^d}\llvert \sm^{I}_{j}\pose{\frac{x-y}{t}}-\sm^{I}_{j}\pose{\frac{x}{t}}\rrvert\frac{2^{j+\ell}}{t^2}\llvert\psi^{\vee}\pose{\frac{x_{d+1}-y_{d+1}+\langle Jx,y\rangle}{t^2 /2^{j+\ell}}}\rrvert,\label{Euclideandifferencepart}
\\
H\pose{\ux,\uy,t}=\frac{1}{t^d}\llvert\sm^{I}_{j}\pose{\frac{x}{t}}\rrvert\frac{2^{j+\ell}}{t^2}\llvert\psi^{\vee}\pose{\frac{x_{d+1}-y_{d+1}+\langle Jx,y\rangle}{t^2 /2^{j+\ell}}}-\psi^{\vee}\pose{\frac{x_{d+1}}{t^2 /2^{j+\ell}}}\rrvert.\label{Heisenbergdifferencepart}
\end{align}
To show \eqref{Hormander110}, it suffices to show that
\begin{align}
\sum_{n\in\Z}\int_{\pose{\widetilde{B}\pose{0,r}}^c}\sup_{t\in[2^{n-1},2^n]}\llvert E\pose{\ux,\uy,t}\rrvert d\ux
\leq C 2^j,\label{Euclideanmv}
\\
\sum_{n\in\Z}\int_{\pose{\widetilde{B}\pose{0,r}}^c}\sup_{t\in[2^{n-1},2^n]}\llvert H\pose{\ux,\uy,t}\rrvert d\ux\leq C 2^j.\label{Heisenbergmv}
\end{align}
Before proving the above two estimates, we recall the following well-known Mean-value Lemma:
\begin{lemma}\label{meanvaluelemma} Let $F\pose{x}$ be a function in Schwartz space $\MS\pose{\R^d}$ and let $y\in\R^d$ with $\llvert y\rrvert\leq1$. Then, for any $N\in\N$, we have
\begin{align}\label{meanvalueestimate}
\llvert F\pose{x-y}-F\pose{x}\rrvert\leq \frac{C_N\llvert y\rrvert}{\pose{1+\llvert x\rrvert}^N}.
\end{align}
\end{lemma}

\begin{proof}[\textbf{Proof of \eqref{Euclideanmv}}] Decompose the sum over $n$ as 
\[
\sum_{n\in\Z}=\sum_{\substack{n\in\Z\\2^{j}\llvert y\rrvert<2^{n-1}}}+\sum_{\substack{n\in\Z\\2^{j}\llvert y\rrvert\geq2^{n-1}}}.
\]
\paragraph{Case \RN{1}: the 1st sum $\sum_{\substack{n\in\Z\\2^{j}\llvert y\rrvert<2^{n-1}}}$} To treat the integrand $\sup_{t\in[2^{n-1},2^n]}\llvert E\pose{\ux,\uy,t}\rrvert$, recall the definition of $\sm^{I}_{j}$ in \eqref{insidesmj}:
\begin{equation*}
\sm^{I}_{j}\pose{x}=2^{jd}\int_{\R^{d+1}}\chi^{\vee}\pose{2^j\pose{x-u}}\psi\pose{x-u}\sm\pose{u}.
\end{equation*}
Then, together with $\llvert\psi\rrvert\leq1$, we have
\begin{align*}
\llvert \sm^{I}_{j}\pose{\frac{x-y}{t}}-\sm^{I}_{j}\pose{\frac{x}{t}}\rrvert&\leq 2^{jd}\int_{\R^{d+1}}\llvert \chi^{\vee}\pose{2^j\pose{\frac{x}{t}-u-\frac{y}{t}}}-\chi^{\vee}\pose{2^j\pose{\frac{x}{t}-u}}\rrvert\sm\pose{u}
\\
&+2^{jd}\int_{\R^{d+1}}\llvert\chi^{\vee}\pose{2^j\pose{x-u}}\rrvert\llvert\psi\pose{\frac{x}{t}-u-\frac{y}{t}}-\psi\pose{\frac{x}{t}-u}\rrvert\sm\pose{u}.
\end{align*}
Since $2^{n-1}\leq t\leq2^n$ and $2^{j}\llvert y\rrvert<2^{n-1}$, it holds that $\llvert 2^j\frac{y}{t}\rrvert<1$. Thus, by applying Lemma \ref{meanvaluelemma} to the function $\chi^{\vee}$, we immediately obtain
\begin{align*}
\llvert \chi^{\vee}\pose{2^j\pose{\frac{x}{t}-u-\frac{y}{t}}}-\chi^{\vee}\pose{2^j\pose{\frac{x}{t}-u}}\rrvert\leq2^j\llvert \frac{y}{t}\rrvert\frac{C_N}{\pose{1+2^j\llvert\frac{x}{t}-u\rrvert}^{N}}.
\end{align*}
Hence, by the above inequality and applying \eqref{measureofsmj} with $N$ sufficiently large, we deduce that
\begin{align*}
 2^{jd}\int_{\R^{d+1}}\llvert \chi^{\vee}\pose{2^j\pose{\frac{x}{t}-u-\frac{y}{t}}}-\chi^{\vee}\pose{2^j\pose{\frac{x}{t}-u}}\rrvert\sm\pose{u}\leq 2^j\llvert \frac{y}{t}\rrvert\frac{C 2^j}{\pose{1+\llvert\frac{x}{t}\rrvert}^{d+1}}.
\end{align*}
Similarly, by applying the Mean-value theorem to $\psi\pose{\frac{x}{t}-u-\frac{y}{t}}-\psi\pose{\frac{x}{t}-u}$ and using \eqref{measureofsmj}, one obtains 
\[2^{jd}\int_{\R^{d+1}}\llvert\chi^{\vee}\pose{2^j\pose{x-u}}\rrvert\llvert\psi\pose{\frac{x}{t}-u-\frac{y}{t}}-\psi\pose{\frac{x}{t}-u}\rrvert\sm\pose{u}\leq2^j\llvert \frac{y}{t}\rrvert\frac{C 2^j}{\pose{1+\llvert\frac{x}{t}\rrvert}^{d+1}}.\]
Therefore, we have
\begin{align}\label{meanvaluetermestimate}
\llvert \sm^{I}_{j}\pose{\frac{x-y}{t}}-\sm^{I}_{j}\pose{\frac{x}{t}}\rrvert\leq2^j\llvert \frac{y}{t}\rrvert\frac{C 2^j}{\pose{1+\llvert\frac{x}{t}\rrvert}^{d+1}}.
\end{align}
Also, since $\psi^{\vee}$ belongs to the Schwartz class, we have
\begin{equation}\label{schwartzheisenberg}
\llvert\psi^{\vee}\pose{\frac{x_{d+1}-y_{d+1}+\langle Jx,y\rangle}{t^2 /2^{j+\ell}}}\rrvert\leq\frac{C_N}{\pose{1+\frac{2^{j+\ell}}{t^2}\llvert x_{d+1}-y_{d+1}+\langle Jx,y\rangle\rrvert}^N}
\end{equation}
for any $N\in\N$.

Now, applying \eqref{meanvaluetermestimate} and \eqref{schwartzheisenberg} to \eqref{Euclideandifferencepart}, we obtain
\begin{align*}
\sup_{t\in[2^{n-1},2^n]}\llvert E\pose{\ux,\uy,t}\rrvert\leq C 2^j\cdot 2^j\llvert \frac{y}{2^n}\rrvert\frac{\pose{1/2^n}^d}{\pose{1+\llvert\frac{x}{2^n}\rrvert}^{d+1}}\frac{\pose{2^{j+\ell}/2^{2n}}}{\pose{1+\frac{2^{j+\ell}}{2^{2n}}\llvert x_{d+1}-y_{d+1}+\langle Jx,y\rangle\rrvert}^2}.
\end{align*}
By inserting this estimate to \eqref{Euclideanmv}, we deduce that
\begin{align*}
\sum_{\substack{n\in\Z\\2^{j}\llvert y\rrvert<2^{n-1}}}\int_{\pose{\widetilde{B}\pose{0,r}}^c}\sup_{t\in[2^{n-1},2^n]}\llvert E\pose{\ux,\uy,t}\rrvert d\ux\leq C 2^j \sum_{\substack{n\in\Z\\2^{j}\llvert y\rrvert<2^{n-1}}}2^j\llvert\frac{y}{2^n}\rrvert\lesssim2^j.
\end{align*}

\paragraph{Case \RN{2}: the 2nd sum $\sum_{\substack{n\in\Z\\2^{j}\llvert y\rrvert\geq2^{n-1}}}$} Combining the condition $2^{j}\llvert y\rrvert\geq2^{n-1}$ with $\llvert y\rrvert<\frac{r}{2^{j+\ell}}$ from \eqref{domainofy}, we have
\begin{equation}\label{rangeofyandr}
2^{n-1}\leq2^j\llvert y\rrvert<\frac{r}{2^{\ell}}\leq r\ \text{and} \sum_{\substack{n\in\Z\\2^{j}\llvert y\rrvert\geq2^{n-1}}}\leq\sum_{\substack{n\in\Z\\ r\geq2^{n-1}}}.
\end{equation}
On the other hand, according to \eqref{domainofx}, the integral over $\pose{\widetilde{B}\pose{0,r}}^c$ can be decomposed as
\[\int_{\pose{\widetilde{B}\pose{0,r}}^c}d\ux=\int_{\llvert x\rrvert>30r}\int_{\R}dx_{d+1}dx+\int_{\llvert x\rrvert\leq 30r}\int_{\llvert x_{d+1}\rrvert>\frac{\pose{30r}^2}{2^{j+\ell}}}dx_{d+1}dx.\]

\paragraph{Case \RN{2}-1: the 1st integral} From the range of $y$ in \eqref{rangeofyandr} and the fact that $\llvert x\rrvert>30r$ in the first integral, we observe that $\llvert x-y\rrvert>\frac{1}{2}\llvert x\rrvert$. Combining this with the upper bound of $\sm_{j}^{I}$ in \eqref{measureofsmj} and using $\llvert \psi\rrvert\leq1$, we obtain the following estimates:
\begin{align*}
\llvert \sm_{j}^{I}\pose{\frac{x-y}{t}}\rrvert, \llvert \sm_{j}^{I}\pose{\frac{x}{t}}\rrvert\leq\frac{C 2^j}{\pose{1+\llvert \frac{x}{t}\rrvert}^{d+1}}\frac{1}{\pose{1+\llvert \frac{x}{t}\rrvert}}.
\end{align*}
Also, note that $\llvert\frac{x}{t}\rrvert>\frac{r}{t}\geq \frac{r}{2^n}$ for $t\in[2^{n-1},2^n]$. Together with the above and \eqref{schwartzheisenberg}, we have
\begin{align*}
\sup_{t\in[2^{n-1},2^n]}\llvert E\pose{\ux,\uy,t}\rrvert\leq C 2^j \llvert \frac{2^n}{r}\rrvert\frac{\pose{1/2^n}^d}{\pose{1+\llvert\frac{x}{2^n}\rrvert}^{d+1}}\frac{\pose{2^{j+\ell}/2^{2n}}}{\pose{1+\frac{2^{j+\ell}}{2^{2n}}\llvert x_{d+1}-y_{d+1}+\langle Jx,y\rangle\rrvert}^2}.
\end{align*}
Therefore, it holds that
\begin{align*}
\sum_{\substack{n\in\Z\\2^{j}\llvert y\rrvert\geq2^{n-1}}}\int_{\llvert x\rrvert>30r}\int_{\R}\sup_{t\in[2^{n-1},2^n]}\llvert E\pose{\ux,\uy,t}\rrvert dx_{d+1}dx\lesssim 2^j \sum_{\substack{n\in\Z\\ r\geq2^{n-1}}}\llvert\frac{2^n}{r}\rrvert\lesssim2^j.
\end{align*}

\paragraph{Case \RN{2}-2: the 2nd integral} From the range of $\uy$ in \eqref{domainofy} and the conditions $\llvert x\rrvert\leq 30r$ and $\llvert x_{d+1}\rrvert>\frac{\pose{30r}^2}{2^{j+\ell}}$ in the second integral, we observe that
\[\llvert y_{d+1}-\langle Jx,y\rangle\rrvert\leq \frac{r^2}{2^{j+\ell}}+\frac{30r^2}{2^{j+\ell}}<\frac{1}{2}\llvert x_{d+1}\rrvert. \]
Using this estimate and \eqref{schwartzheisenberg}, we obtain
\begin{align*}
\llvert\psi^{\vee}\pose{\frac{x_{d+1}-y_{d+1}+\langle Jx,y\rangle}{t^2 /2^{j+\ell}}}\rrvert\leq \frac{C}{\pose{1+\frac{2^{j+\ell}}{t^2}\llvert x_{d+1}\rrvert}^2}\frac{1}{\pose{1+\frac{2^{j+\ell}}{t^2}\llvert x_{d+1}\rrvert}}.
\end{align*}
Also, since $\llvert x_{d+1}\rrvert>\frac{\pose{30r}^2}{2^{j+\ell}}$, for $t\in[2^{n-1},2^n]$, it holds that $\frac{2^{j+\ell}}{t^2}\llvert x_{d+1}\rrvert\geq \frac{r^2}{2^{2n}}$. Thus, we have
\begin{align*}
\sup_{t\in[2^{n-1},2^n]}\llvert E\pose{\ux,\uy,t}\rrvert&\leq C 2^j\llvert \frac{2^{2n}}{r^2}\rrvert\frac{\pose{1/2^n}^d}{\pose{1+\llvert\frac{x-y}{2^n}\rrvert}^{d+1}}\frac{\pose{2^{j+\ell}/2^{2n}}}{\pose{1+\frac{2^{j+\ell}}{2^{2n}}\llvert x_{d+1}\rrvert}^2}
\\
&+C 2^j\llvert \frac{2^{2n}}{r^2}\rrvert\frac{\pose{1/2^n}^d}{\pose{1+\llvert\frac{x}{2^n}\rrvert}^{d+1}}\frac{\pose{2^{j+\ell}/2^{2n}}}{\pose{1+\frac{2^{j+\ell}}{2^{2n}}\llvert x_{d+1}\rrvert}^2}.
\end{align*}
Since $y$ is fixed and $\int \frac{\pose{1/2^n}^d}{\pose{1+\llvert\frac{x}{2^n}\rrvert}^{d+1}}dx\lesssim 1$, we obtain
\begin{align*}
\sum_{\substack{n\in\Z\\2^{j}\llvert y\rrvert\geq2^{n-1}}}\int_{\llvert x\rrvert\leq 30r}\int_{\llvert x_{d+1}\rrvert>\frac{\pose{30r}^2}{2^{j+\ell}}}\sup_{t\in[2^{n-1},2^n]}\llvert E\pose{\ux,\uy,t}\rrvert dx_{d+1}dx\lesssim 2^j \sum_{\substack{n\in\Z\\ r\geq2^{n-1}}}\llvert\frac{2^{2n}}{r^2}\rrvert\lesssim2^j.
\end{align*}
Therefore, we conclude that \eqref{Euclideanmv} holds.
\end{proof}

\begin{proof}[\textbf{Proof of \eqref{Heisenbergmv}}] Recall the definition of $\sm^{I}_{j}$ in \eqref{insidesmj}:
\begin{equation*}
\sm^{I}_{j}\pose{\frac{x}{t}}=2^{jd}\int_{\R^{d+1}}\chi^{\vee}\pose{2^j\pose{\frac{x}{t}-u}}\psi\pose{\frac{x}{t}-u}\sm\pose{u}.
\end{equation*}
Since $\llvert u\rrvert=1$ and by the support condition $\llvert \frac{x}{t}-u\rrvert<2$, we deduce that 
\begin{equation}\label{xrangewitht}
\llvert \frac{x}{t}\rrvert\leq 3.
\end{equation}
 Also, from now on, we assume $t\in[2^{n-1},2^n]$. Now, we decompose the sum over $n$ as 
\[
\sum_{n\in\Z}=\sum_{\substack{n\in\Z \\ 2^n<10r}}+\sum_{\substack{n\in\Z \\ 2^n\geq10r}}.
\]
\paragraph{Case \RN{1}: the 1st sum $\sum_{\substack{n\in\Z \\ 2^n<10r}}$}From $\llvert x\rrvert\leq 3t\leq3\cdot2^n<30r$ and \eqref{domainofx}, one can see that 
\[\int_{\pose{\widetilde{B}\pose{0,r}}^c}d\ux=\int_{\llvert x\rrvert\leq 30r}\int_{\llvert x_{d+1}\rrvert>\frac{\pose{30r}^2}{2^{j+\ell}}}dx_{d+1}dx.\] 
Also, from the range of $\uy$ in \eqref{domainofy} combined with $\llvert x\rrvert\leq 30r$ and $\llvert x_{d+1}\rrvert>\frac{\pose{30r}^2}{2^{j+\ell}}$, we deduce that 
\[\llvert y_{d+1}-\langle Jx,y\rangle\rrvert\leq \frac{r^2}{2^{j+\ell}}+\frac{30r^2}{2^{j+\ell}}<\frac{1}{2}\llvert x_{d+1}\rrvert. \]
Together with \eqref{schwartzheisenberg}, the following estimates hold:
\begin{align*}
\llvert\psi^{\vee}\pose{\frac{x_{d+1}-y_{d+1}+\langle Jx,y\rangle}{t^2 /2^{j+\ell}}}\rrvert,\ \llvert \psi^{\vee}\pose{\frac{x_{d+1}}{t^2 /2^{j+\ell}}}\rrvert\leq \frac{C}{\pose{1+\frac{2^{j+\ell}}{t^2}\llvert x_{d+1}\rrvert}^2} \frac{1}{\pose{1+\frac{2^{j+\ell}}{t^2}\llvert x_{d+1}\rrvert}}.
\end{align*}
Again, since $\llvert x_{d+1}\rrvert>\frac{\pose{30r}^2}{2^{j+\ell}}$, for $t\in[2^{n-1},2^n]$, it holds that $\frac{2^{j+\ell}}{t^2}\llvert x_{d+1}\rrvert\geq \frac{r^2}{2^{2n}}$. Thus, we have
\begin{align*}
\sup_{t\in[2^{n-1},2^n]}\llvert H\pose{\ux,\uy,t}\rrvert\leq C 2^j\llvert \frac{2^{2n}}{r^2}\rrvert\frac{\pose{1/2^n}^d}{\pose{1+\llvert\frac{x}{2^n}\rrvert}^{d+1}}\frac{\pose{2^{j+\ell}/2^{2n}}}{\pose{1+\frac{2^{j+\ell}}{2^{2n}}\llvert x_{d+1}\rrvert}^2}.
\end{align*}
Therefore, we can control the first sum $\sum_{\substack{n\in\Z \\ 2^n<10r}}$ by 
\begin{align*}
\sum_{\substack{n\in\Z \\ 2^n<10r}} \int_{\llvert x\rrvert\leq 30r}\int_{\llvert x_{d+1}\rrvert>\frac{\pose{30r}^2}{2^{j+\ell}}}\sup_{t\in[2^{n-1},2^n]}\llvert H\pose{\ux,\uy,t}\rrvert dx_{d+1}dx\lesssim 2^j \sum_{\substack{n\in\Z \\ 2^n<10r}} \llvert\frac{2^{2n}}{r^2}\rrvert\lesssim2^j.
\end{align*}

\paragraph{Case \RN{2}: the 2nd sum $\sum_{\substack{n\in\Z \\ 2^n\geq10r}}$}Note that $10r\leq 2^n\leq 2t$. Then, together with the range of $\uy$ in \eqref{domainofy} and $\llvert x\rrvert\leq3t$ in \eqref{xrangewitht}, one can observe that
\begin{equation}\label{meanvalueheisenbergsize}
\llvert y_{d+1}-\langle Jx,y\rangle\rrvert<\frac{r^2}{2^{j+\ell}}+\frac{3tr}{2^{j+\ell}}\leq\frac{4tr}{2^{j+\ell}}\leq\frac{t^2}{2^{j+\ell}}.
\end{equation}
Thus, we deduce that \[\frac{2^{j+\ell}}{t^2}\llvert y_{d+1}-\langle Jx,y\rangle\rrvert\leq1.\] By utilizing Lemma \ref{meanvaluelemma}, we obtain
\begin{align*}
\llvert\psi^{\vee}\pose{\frac{x_{d+1}-y_{d+1}+\langle Jx,y\rangle}{t^2 /2^{j+\ell}}}-\psi^{\vee}\pose{\frac{x_{d+1}}{t^2 /2^{j+\ell}}}\rrvert&\leq\frac{2^{j+\ell}}{t^2}\llvert y_{d+1}-\langle Jx,y\rangle\rrvert\frac{C}{\pose{1+\frac{2^{j+\ell}}{t^2}\llvert x_{d+1}\rrvert}^2} \\
&\leq\frac{4r}{t}\frac{C}{\pose{1+\frac{2^{j+\ell}}{t^2}\llvert x_{d+1}\rrvert}^2} .
\end{align*}
The second line of the above comes from $\llvert y_{d+1}-\langle Jx,y\rangle\rrvert<\frac{4tr}{2^{j+\ell}}$ in \eqref{meanvalueheisenbergsize}. 
Then, it holds that
\[\sup_{t\in[2^{n-1},2^n]}\llvert H\pose{\ux,\uy,t}\rrvert\leq C 2^j \frac{r}{2^n}\frac{\pose{1/2^n}^d}{\pose{1+\llvert\frac{x}{2^n}\rrvert}^{d+1}}\frac{\pose{2^{j+\ell}/2^{2n}}}{\pose{1+\frac{2^{j+\ell}}{2^{2n}}\llvert x_{d+1}\rrvert}^2}.\]
Therefore, we have
\[\sum_{\substack{n\in\Z \\ 2^n\geq10r}}\int_{\pose{\widetilde{B}\pose{0,r}}^c}\sup_{t\in[2^{n-1},2^n]}\llvert H\pose{\ux,\uy,t}\rrvert d\ux\leq C  2^j \sum_{\substack{n\in\Z \\ 2^n\geq10r}}\frac{r}{2^n}\lesssim2^j,\]
which conclude the proof of \eqref{Heisenbergmv}.
\end{proof}

\section{Some Reductions for the $L^2$ Estimate for $\MM_{j,\ell}^{I}$}\label{sectionreductionl2}
Now, we shall show that \eqref{l2estimate} holds:
\begin{equation}\label{L2estimate}
\llVert \MM_{j,\ell}^{I}\rrVert_{L^{2}\pose{\R^{d+1}}\rightarrow L^2\pose{\R^{d+1}}}\leq C 2^{-j\pose{\frac{d-2}{2}}}2^{-\ell\pose{\frac{d}{2}}}.
\end{equation}
Recall that $\MM_{j,\ell}^{I}$ is defined by:
\begin{align*}
\varphi_{j,\ell}\pose{\xi_{d+1}}&=\left\{\begin{array}{ll} \psi\pose{\frac{\xi_{d+1}}{2^j}}\text{if}\ \ell=0, \\  \chi\pose{\frac{\xi_{d+1}}{2^{j+\ell}}}\text{if}\ \ell\geq1,\end{array}\right. \nonumber
\\
\MA^{I}_{j,\ell}\pose{f}\pose{\ux,t}&=\int_{\R^{d+1}} e^{2\pi i\ux\cdot\uxi}\widehat{\sm_{j}^{I}}\pose{t\pose{\xi+\xi_{d+1}Jx}}\varphi_{j,\ell}\pose{t^2\xi_{d+1}}\widehat{f}\pose{\xi,\xi_{d+1}}d\xi d\xi_{d+1},
\\
\MM_{j,\ell}^{I}\pose{f}\pose{\ux}&=\sup_{t>0}\llvert \MA_{j,\ell}^{I}\pose{f}\pose{\ux,t}\rrvert.
\end{align*}
We first decompose $\widehat{\sm_{j}^{I}}$ as the main term $\widehat{d\mu_j}$ and the tail term $\widehat{d\nu_j}$:
\begin{gather}
\widehat{\sm_{j}^{I}}\pose{\xi}=\widehat{d\mu_j}\pose{\xi}+\widehat{d\nu_j}\pose{\xi}\ \text{where}\nonumber
\\
\widehat{d\mu_j}\pose{\xi}=\chi_0\pose{\frac{\xi}{2^j}}\widehat{\sm_{j}^{I}}\pose{\xi}\ \text{and}\ \widehat{d\nu_j}\pose{\xi}=\chi_0^c\pose{\frac{\xi}{2^j}}\widehat{\sm_{j}^{I}}\pose{\xi}.\label{insideandoutside}
\end{gather} Here, $\chi_0\pose{\xi}$ is a smooth function supported in $\{\frac{1}{4}\leq\llvert \xi\rrvert\leq4\}$ with $0\leq\chi_0\leq1$ and $\chi_0^c=1-\chi_0$. Define the corresponding maximal functions $ M_{j,\ell}$ and $\MV_{j,\ell}$ by
\begin{align}
\MS_{j,\ell}\pose{f}\pose{\ux,t}&=\int_{\R^{d+1}} e^{2\pi i\ux\cdot\uxi}\widehat{d\mu_{j}}\pose{t\pose{\xi+\xi_{d+1}Jx}}\varphi_{j,\ell}\pose{t^2\xi_{d+1}}\widehat{f}\pose{\xi,\xi_{d+1}}d\xi d\xi_{d+1}, \label{inside}\\
 M_{j,\ell}\pose{f}\pose{\ux}&=\sup_{t>0}\llvert \MS_{j,\ell}\pose{f}\pose{\ux,t}\rrvert, \nonumber\\
\MS_{j,\ell}^c\pose{f}\pose{\ux,t}&=\int_{\R^{d+1}} e^{2\pi i\ux\cdot\uxi}\widehat{d\nu_{j}}\pose{t\pose{\xi+\xi_{d+1}Jx}}\varphi_{j,\ell}\pose{t^2\xi_{d+1}}\widehat{f}\pose{\xi,\xi_{d+1}}d\xi d\xi_{d+1}, \label{outside}\\
\MV_{j,\ell}\pose{f}\pose{\ux}&=\sup_{t>0}\llvert \MS_{j,\ell}^c\pose{f}\pose{\ux,t}\rrvert.\nonumber
\end{align}
To show \eqref{L2estimate}, it suffices to prove the followings:
\begin{align}
\llVert  M_{j,\ell}\rrVert_{L^2\pose{\R^{d+1}}\rightarrow L^2\pose{\R^{d+1}}}&\leq C  2^{-j\pose{\frac{d-2}{2}}}2^{-\ell\pose{\frac{d}{2}}}\ \text{and}\label{l2main} \\
\llVert \MV_{j,\ell}\rrVert_{L^2\pose{\R^{d+1}}\rightarrow L^2\pose{\R^{d+1}}}&\leq C  2^{-j\pose{\frac{d-2}{2}}}2^{-\ell\pose{\frac{d}{2}}}.\label{l2tail}
\end{align}
We prove the tail estimate \eqref{l2tail} by using the same method to treat $\MM_1$ (see the Appendix).
In this section and Section 5, we focus on proving the main estimate \eqref{l2main}.

\subsection{Majorizing $L^{\infty}(\R_{+})$ by $L^{2}_{1/2}(\R_{+})$ via Sobolev Embedding}\label{subsectionsobolev}

To handle $L^{\infty}\pose{dt}$, we first define
\begin{align}\label{averagemainS}
\MS_{j,\ell,k}\pose{f}\pose{\ux,t}&=\chi(2^k t)\MS_{j,\ell}\pose{f}\pose{\ux,t}\nonumber\\
&=\chi(2^k t)\int_{\R^{d+1}} e^{2\pi i\ux\cdot\uxi}\widehat{d\mu_{j}}\pose{t\llvert\xi+\xi_{d+1}Jx\rrvert}\varphi_{j,\ell}\pose{t^2\xi_{d+1}}\widehat{f}\pose{\xi,\xi_{d+1}}d\xi d\xi_{d+1}
\end{align}
for $k\in\Z$. Then, it holds that
\begin{equation*}
\llvert M_{j,\ell}\pose{f}\pose{\ux}\rrvert\leq\pose{\sum_{k\in\Z}\sup_{t>0}\llvert\MS_{j,\ell,k}\pose{f}\pose{\ux,t}\rrvert^2}^{\frac{1}{2}}.
\end{equation*}
Using the standard Sobolev embedding argument to the right-hand side of the above, one can see that
\begin{align}\label{sobolev1}
\int_{\R^{d+1}}\llvert M_{j,\ell}\pose{f}\pose{\ux}\rrvert^2 d\ux&\leq\sum_{k\in\Z}\int_{\R^{d+1}}\int_{0}^{\infty}\pose{\pd_{t}\MS_{j,\ell,k}\pose{f}\pose{\ux,t}}\pose{\overline{\MS_{j,\ell,k}\pose{f}\pose{\ux,t}}}dtd\ux\\
\nonumber &+\sum_{k\in\Z}\int_{\R^{d+1}}\int_{0}^{\infty}\pose{\MS_{j,\ell,k}\pose{f}\pose{\ux,t}}\pose{\pd_{t}\overline{\MS_{j,\ell,k}\pose{f}\pose{\ux,t}}}dtd\ux.
\end{align} 
Now, we focus on the first term of the right-hand side of \eqref{sobolev1}:
\begin{equation*}
\int_{\R^{d+1}}\int_{0}^{\infty}\pose{\pd_{t}\MS_{j,\ell,k}\pose{f}\pose{\ux,t}}\pose{\overline{\MS_{j,\ell,k}\pose{f}\pose{\ux,t}}}dtd\ux.
\end{equation*}
To treat the $t$ derivative of the amplitude $\chi(2^k t)\widehat{d\mu_{j}}\pose{t\pose{\xi+\xi_{d+1}Jx}}\varphi_{j,\ell}\pose{t^2\xi_{d+1}}$ in \eqref{averagemainS}, one can observe that
\begin{align*}
\widehat{d\mu_j}\pose{t\xi}&=\chi_0\pose{\frac{t\xi}{2^j}}\pose{\int_{\R^d}\chi\pose{\frac{t\xi-\eta}{2^j}}\widehat{\psi}\pose{\eta}d\eta}\widehat{\sm}\pose{t\llvert \xi\rrvert},
\\
\pd_{t}\left[\chi_0\pose{\frac{t\xi}{2^j}}\right]&=\frac{\llvert\xi\rrvert}{2^j}\widetilde{\chi_0}\pose{\frac{t\xi}{2^j}}
\\
\pd_{t}\left[\widehat{\sm}\pose{t\llvert\xi\rrvert}\right]&=2\pi i\llvert\xi\rrvert\int_{S^{d-1}}e^{2\pi i t\xi\cdot x}\left(\frac{\xi}{\llvert\xi\rrvert}\cdot x\right)\sm\pose{x},
\\
\pd_{t}\left[\int_{\R^d}\chi\pose{\frac{t\xi-\eta}{2^j}}\widehat{\psi}\pose{\eta}d\eta\right]&=\frac{\llvert\xi\rrvert}{2^j}\pose{\int_{\R^d}\widetilde{\chi}\pose{\frac{t\xi-\eta}{2^j}}\widehat{\psi}\pose{\eta}d\eta},
\\
\pd_{t}\left[\varphi_{j,\ell}\pose{t^2\xi_{d+1}}\right]&=\left\{\begin{array}{ll} 2\llvert\frac{t\xi_{d+1}}{2^j}\rrvert\widetilde{\psi}\pose{\frac{\xi_{d+1}}{2^j}}\text{if}\ \ell=0, \\  2\llvert\frac{t\xi_{d+1}}{2^{j+\ell}}\rrvert\widetilde{\chi}\pose{\frac{\xi_{d+1}}{2^{j+\ell}}}\text{if}\ \ell\geq1,\end{array}\right.
\end{align*}
where $\widetilde{\psi}\pose{\xi}$, $\widetilde{\chi}\pose{\xi}$ and $\widetilde{\chi_0}\pose{\xi}$ are smooth functions supported in $\{\llvert \xi\rrvert\lesssim 1\}$, $\{\llvert \xi\rrvert\sim 1\}$, and $\{\frac{1}{4}\leq\llvert \xi\rrvert\leq4\}$, respectively. Then, from $t\sim2^{-k}$, $\llvert\xi+\xi_{d+1}Jx\rrvert\sim2^{j+k}$, and $\llvert\frac{t\xi_{d+1}}{2^{j+\ell}}\rrvert\lesssim 2^{k}$ in \eqref{averagemainS}, it follows that
\begin{align*}
\llvert\pd_{t} \MS_{j,\ell,k}\pose{f}\pose{\ux,t}\rrvert\lesssim 2^{j+k}\llvert\MS_{j,\ell,k}\pose{f}\pose{\ux,t}\rrvert.
\end{align*}
with slight modifications to $\widehat{\sm}$, $\psi$, and $\chi$. This implies a Sobolev embedding in $t$ such that
\begin{align}\label{majorsobolev}
\sup_{t>0}\llvert\MS_{j,\ell,k}\pose{f}\pose{\ux,t}\rrvert^2&\leq 2\int_{0}^{\infty}\llvert\pd_{t}\MS_{j,\ell,k}\pose{f}\pose{\ux,t}\rrvert\llvert \MS_{j,\ell,k}\pose{f}\pose{\ux,t}\rrvert dt\\
&\lesssim 2^{j+k}\int_{0}^{\infty}\llvert \MS_{j,\ell,k}\pose{f}\pose{\ux,t}\rrvert^2 dt.\nonumber
\end{align}
By \eqref{sobolev1} and \eqref{majorsobolev}, we have 
\[
\llVert  M_{j,\ell}\pose{f}\rrVert_{L^{2}\pose{\R^{d+1}}}\lesssim 2^{\frac{j}{2}}\llVert \pose{\sum_{k\in\Z} \llvert 2^{\frac{k}{2}}\MS_{j,\ell,k}\pose{f}\rrvert^2}^{1/2} \rrVert_{L^{2}\pose{\R^{d+1}\times \R_{+}}}. 
\]
Therefore, in order to prove \eqref{l2main}, it suffices to show the following square sum estimate
\begin{equation}\label{squaresum1}
\llVert\pose{\sum_{k\in\Z} \llvert 2^{\frac{k}{2}}\MS_{j,\ell,k}\pose{f}\rrvert^2 }^{1/2} \rrVert_{L^{2}\pose{\R^{d+1}\times \R_{+}}}
\lesssim 2^{-j\pose{\frac{d-1}{2}}}2^{-\ell\pose{\frac{d}{2}}}\llVert f\rrVert_{L^2 \pose{\R^{d+1}}}.
\end{equation}

\subsection{Reduction to Uniform $L^{2}$ Estimates in $\lambda$}
We denote the Euclidean Fourier transform of $f$ as the unified notation $\widehat{f}$ with respect to its various domains.
For instance, if $f\in L^2\pose{\R^{n_1}}$ and $g\in L^2\pose{\R^{n_2}}$, then $\widehat{f}$ and $\widehat{g}$ are defined as
\begin{equation*}
\widehat{f}\pose{\xi}=\int_{\R^{n_1}}e^{-2\pi i x\cdot\xi}f\pose{x}dx\ \text{and}\
\widehat{g}\pose{\eta}=\int_{\R^{n_2}}e^{-2\pi i y\cdot\eta}g\pose{y}dy.
\end{equation*}
Now, we begin with a dimension reducing argument via the Plancherel theorem.

For a scalar $\lambda\in\R$ and $g\in L^2\pose{\R^d}$, we define an operator $\MS_{j,\ell,k}^{\lambda}$ by
\begin{equation}\label{reduced1}
\MS_{j,\ell,k}^{\lambda}\pose{g}\pose{x,t}=\int_{\R^{d}} e^{2\pi i x\cdot\xi}\chi(2^k t)\widehat{d\mu_{j}}\pose{t\pose{\xi+\lambda Jx}}\varphi_{j,\ell}\pose{t^2\lambda}\widehat{g}\pose{\xi}d\xi.
\end{equation}

\begin{proposition}\label{reductionuniformL2}
To show \eqref{squaresum1}, it suffices to prove the following estimate:
\begin{equation}\label{squaresum2}
\sup_{\lambda\in\R}\llVert\pose{\sum_{k\in\Z} \llvert 2^{\frac{k}{2}}\MS_{j,\ell,k}^{\lambda}\pose{\cdot}\rrvert^2 }^{1/2} \rrVert_{L^{2}\pose{\R^{d}}\rightarrow L^{2}\pose{\R^{d}\times \R_{+}}}
\leq C 2^{-j\pose{\frac{d-1}{2}}}2^{-\ell\pose{\frac{d}{2}}},
\end{equation}
where $C$ is independent of $\lambda$.
\end{proposition}
\begin{proof}
For $f\in L^2\pose{\R^{d+1}}$, by letting $\xi_{d+1}=\lambda$, one can see that 
\begin{equation}\label{plancherelpart1}
\MS_{j,\ell,k}\pose{f}\pose{x,x_{d+1},t}=\int_{\R} e^{2\pi i x_{d+1}\lambda}\MS_{j,\ell,k}^{\lambda}\pose{f\pose{\cdot,\lambda}}\pose{x,t}d\lambda 
\end{equation}
where $f\pose{x,\lambda}=\int_{\R^d}e^{2\pi i x\cdot \xi}\widehat{f}\pose{\xi,\lambda}d\xi$ for each fixed $\lambda\in\R$. Then, by \eqref{plancherelpart1} and the Plancherel theorem for $\lambda$ and $x_{d+1}$, it holds that
\begin{align*}
\int_{\R^d}\int_{\R}\sum_{k\in\Z} \llvert 2^{\frac{k}{2}}\MS_{j,\ell,k}\pose{f}\pose{x,x_{d+1},t}\rrvert^2 dx_{d+1}dx
=\int_{\R^d}\int_{\R}\sum_{k\in\Z} \llvert 2^{\frac{k}{2}}\MS_{j,\ell,k}^{\lambda}\pose{f\pose{\cdot,\lambda}}\pose{x,t}\rrvert^2 d\lambda dx.
\end{align*}
Utilizing \eqref{squaresum2} to the above, we have 
\begin{align*}
\int_{\R}\int_{0}^{\infty}\int_{\R^d}\sum_{k\in\Z} \llvert 2^{\frac{k}{2}}\MS_{j,\ell,k}^{\lambda}\pose{f\pose{\cdot,\lambda}}\pose{x,t}\rrvert^2 dxdtd\lambda&\leq C^22^{-j\pose{d-1}}2^{-\ell d}\int_{\R}\int_{\R^d}\llvert f\pose{x,\lambda}\rrvert^2 dxd\lambda \\
&=C^22^{-j\pose{d-1}}2^{-\ell d}\int_{\R^d}\int_{\R}\llvert f\pose{x,x_{d+1}}\rrvert^2 dx_{d+1}dx\\
&=C^22^{-j\pose{d-1}}2^{-\ell d}\llVert f\rrVert_{L^2\pose{\R^{d+1}}}^2,
\end{align*}
which implies \eqref{squaresum1}.
\end{proof}
Our goal now is to prove \eqref{squaresum2}. First of all, we claim that it suffices to consider $\sup_{\lambda>0}$ instead of $\sup_{\lambda\in\R}$. If $\lambda=0$ in \eqref{reduced1}, then we only need to consider the case of $\ell=0$ from the support condition of $\varphi_{j+\ell}$. Also, for $f\in L^2\pose{\R^{d}}$, it holds that
\[\MS_{j,0,k}^{0}\pose{f}\pose{x,t}=\int_{\R^{d}} e^{2\pi i x\cdot\xi}\chi(2^k t)\widehat{d\mu_{j}}\pose{t\xi}\widehat{f}\pose{\xi}d\xi\] 
which corresponds to the Euclidean case. Thus \eqref{squaresum2} holds from the decay rate $2^{-j\pose{\frac{d-1}{2}}}$ of $\widehat{\sm}$ and the disjointness of supports of $\widehat{d\mu_j}\pose{2^{-k}\cdot}\sim\widehat{\sm_j}\pose{2^{-k}\llvert\cdot\rrvert}$ in $k$ (See Stein's result \cite{Stein} or Bourgain's result \cite{Bourgain1986AveragesIT} when $d=2$). In the proof of $\eqref{squaresum2}$ with $\sup_{\lambda\neq0}$, we will do not use the difference of sign of $\lambda$. Thus we can prove the $\lambda<0$ case by the same way to prove the case of $\lambda>0$. Therefore, to prove \eqref{squaresum2}, it suffices to show that
\begin{equation}\label{squaresum3}
\sup_{\lambda>0}\llVert\pose{\sum_{k\in\Z} \llvert 2^{\frac{k}{2}}\MS_{j,\ell,k}^{\lambda}\pose{f}\rrvert^2 }^{1/2} \rrVert_{ L^{2}\pose{\R^{d}\times \R_{+}}}
\leq C 2^{-j\pose{\frac{d-1}{2}}}2^{-\ell\pose{\frac{d}{2}}}\llVert f\rrVert_{L^2\pose{\R^d}},
\end{equation}
where $C$ is independent of $\lambda$. From now on, all implicit constants $C$ are independent of $\lambda$.

\subsection{Asymptotic Formula for $\sm$}\label{secasymptotic}
Recall that
\[\widehat{d\mu_j}\pose{\xi}=\chi_0\pose{\frac{\xi}{2^j}}\pose{\int_{\R^d}\chi\pose{\frac{\xi-\eta}{2^j}}\widehat{\psi}\pose{\eta}d\eta}\widehat{\sm}\pose{\llvert\xi\rrvert}.\]
For a convenience, define
\begin{equation}\label{perturbterm}
\widetilde{\chi}_j\pose{\xi}=\int_{\R^d}\chi\pose{\frac{\xi-\eta}{2^j}}\widehat{\psi}\pose{\eta}d\eta.
\end{equation}
Since $\chi_0\pose{\frac{\xi}{2^j}}\widehat{\sm}\pose{\llvert\xi\rrvert}$ is supported in $\llvert \xi\rrvert\geq1$ for $j\geq2$, one can apply the asymptotics of the Bessel function $J_{\frac{d-2}{2}}$ (see, e.g., \cite[pp. 338 and 356]{steinhm}) to obtain
\begin{align}\label{asymptoticsm}
\widehat{d\mu_{j}}\pose{\xi}&=\sum_{n=0}^{d}\pose{a_{n}e^{2\pi i \llvert\xi\rrvert}+b_{n}e^{-2\pi i \llvert\xi\rrvert}}\frac{\pose{2\pi}^{-n}\chi_{n}\pose{\frac{\xi}{2^j}}\widetilde{\chi}_j\pose{\xi}}{2^{j\pose{\frac{d-1}{2}+n}}}+\epsilon_{j}\pose{\xi}
\end{align}
where $\chi_{n}\pose{\xi}=\chi_0\pose{\xi}\llvert\xi\rrvert^{-\pose{\frac{d-1}{2}+n}}$ with $a_{0}=\pose{\frac{\pi}{2}}^{-\frac{1}{2}}e^{-i\pose{\frac{d-1}{4}\pi}}$ and $b_{0}=\pose{\frac{\pi}{2}}^{-\frac{1}{2}}e^{i\pose{\frac{d-1}{4}\pi}}$, and the error term $\epsilon_{j}\pose{\xi}$ is defined by 
\begin{align}\label{errorterm}
\epsilon_{j}\pose{\xi}&=\pose{\widehat{d\mu_{j}}\pose{\xi}-\sum_{n=0}^{d}\pose{a_{n}e^{2\pi i \llvert\xi\rrvert}+b_{n}e^{-2\pi i \llvert\xi\rrvert}}\frac{\pose{2\pi}^{-n}\chi_{n}\pose{\frac{\xi}{2^j}}\widetilde{\chi}_j\pose{\xi}}{2^{j\pose{\frac{d-1}{2}+n}}}}\\
&=\pose{\widehat{\sm}\pose{\xi}-\sum_{n=0}^{d}\pose{a_{n}e^{2\pi i \llvert\xi\rrvert}+b_{n}e^{-2\pi i \llvert\xi\rrvert}}\pose{2\pi}^{-n}\llvert\xi\rrvert^{-\pose{\frac{d-1}{2}+n}}}\widetilde{\chi}_j\pose{\xi}\chi_0\pose{\frac{\xi}{2^j}}.\nonumber
\end{align}
Note that the error term satisfies $\epsilon_{j}\pose{\xi}=O\pose{\frac{\chi\pose{\frac{\xi}{2^j}}}{2^{j\pose{\frac{d-1}{2}+d+\frac{1}{2}}}}}$ since $\llvert \widetilde{\chi}_j\pose{\xi}\rrvert\leq C $. Also, the only reason for using $\chi_0$ is to gurantee that $\llvert \xi\rrvert\geq1$ so that we can apply the asymptotics of $\widehat{\sm}$. Thus, from now on, we will simply use $\chi$ instead of $\chi_0$.

Now, we focus on the dominating term $n=0$ in $\sum_{n=0}^{d}$ and on the error term separately. Define
\begin{align*}
\tilde{\MS}_{j,\ell,k}^{\lambda}\pose{f}\pose{x,t}&=2^{\frac{k}{2}}\int_{\R^d} e^{2\pi i \left[x\cdot\xi+t\llvert\xi+\lambda Jx\rrvert\right]} \chi\pose{2^{k}t}\frac{\chi\pose{\frac{t\llvert\xi+\lambda Jx\rrvert}{2^{j}}}\widetilde{\chi}_j\pose{t\pose{\xi+\lambda Jx}}}{2^{j\pose{\frac{d-1}{2}}}}\varphi_{j,\ell}\pose{t^2\lambda}\widehat{f}\pose{\xi}d\xi,
\\
E_{j,\ell,k}^{\lambda}\pose{f}\pose{x,t}&=2^{\frac{k}{2}}\int_{\R^d} e^{2\pi i x\cdot\xi} \chi\pose{2^{k}t}\epsilon_{j}\pose{t\pose{\xi+\lambda Jx}}\varphi_{j,\ell}\pose{t^2\lambda}\widehat{f}\pose{\xi}d\xi. 
\end{align*}

Observing that the first term in the right-hand side of \eqref{asymptoticsm} dominates the other terms corresponding to $n=1,\cdots,d$, we obtain the following pointwise bound:
\[\llvert 2^{\frac{k}{2}}\MS_{j,\ell,k}^{\lambda}\pose{f}\pose{x,t}\rrvert\lesssim \llvert \tilde{\MS}_{j,\ell,k}^{\lambda}\pose{f}\pose{x,t}\rrvert+\llvert E_{j,\ell,k}^{\lambda}\pose{f}\pose{x,t}\rrvert.\]
Hence, \eqref{squaresum3} follows from the two estimates:
\begin{flalign}
\sup_{\lambda>0}\llVert\pose{\sum_{k\in\Z}\llvert \tilde{\MS}_{j,\ell,k}^{\lambda}\pose{f}\rrvert^2}^{1/2}\rrVert_{L^{2}\pose{\R^d\times\R_{+}}}\lesssim 2^{-j\pose{\frac{d-1}{2}}}2^{-\ell\pose{\frac{d}{2}}}\llVert f\rrVert_{L^{2}\pose{\R^d}},\label{squareL2main}
\\
\sup_{\lambda>0}\llVert\pose{\sum_{k\in\Z}\llvert E_{j,\ell,k}^{\lambda}\pose{f}\rrvert^2}^{1/2}\rrVert_{L^{2}\pose{\R^d\times\R_{+}}}\lesssim 2^{-j\pose{\frac{d-1}{2}}}2^{-\ell\pose{\frac{d}{2}}}\llVert f\rrVert_{L^{2}\pose{\R^d}}.\label{squareL2error}
\end{flalign}
Since the error estimate \eqref{squareL2error} can be handled more easily by the similar method as the main estimate—thanks to the strong decay rate $2^{-j\pose{\frac{d-1}{2}+d+\frac{1}{2}}}$ in the error term—we shall focus on proving \eqref{squareL2main} and omit the details for the error estimate.

\subsection{Dilation and More Reductions}
We shall divide the square sum on the left-hand-side of \eqref{squareL2main} as the dichotomy of $\ell$ in \eqref{pieceofd+1}:
 \begin{align*}
\varphi_{j,\ell}\pose{\xi_{d+1}}&=\left\{\begin{array}{ll} \psi\pose{\frac{\xi_{d+1}}{2^j}}\text{if}\ \ell=0, \\  \chi\pose{\frac{\xi_{d+1}}{2^{j+\ell}}}\text{if}\ \ell\geq1.\end{array}\right. 
\end{align*}
For a fixed $\lambda>0$ and $j,\ell>0$, there are at most 5 different $k$'s satisfying the support condition of $\chi\pose{2^k t}\chi\pose{\frac{t^2\lambda}{2^{j+\ell}}}\sim\chi\pose{2^k t}\chi\pose{\frac{\lambda}{2^{j+2k+\ell}}}$. Thus, we have 
\begin{align*}
\llVert\pose{\sum_{k\in\Z}\llvert \tilde{\MS}_{j,\ell,k}^{\lambda}\pose{f}\rrvert^2}^{1/2}\rrVert_{L^{2}\pose{\R^d\times\R_{+}}}\lesssim\sup_{k\in\Z} \llVert\tilde{\MS}_{j,\ell,k}^{\lambda}\pose{f}\rrVert_{L^{2}\pose{\R^d\times\R_{+}}}.
\end{align*}
Hence, to show \eqref{squareL2main}, it suffices to prove that
\begin{align}
\sup_{\lambda>0}\sup_{k\in\Z} \llVert\tilde{\MS}_{j,\ell,k}^{\lambda}\pose{f}\rrVert_{L^{2}\pose{\R^d\times\R_{+}}}&\lesssim 2^{-j\pose{\frac{d-1}{2}}}2^{-\ell\pose{\frac{d}{2}}}\llVert f\rrVert_{L^{2}\pose{\R^d}}\ \text{for}\ \ell>0,\label{mainL2positive}
\\
\sup_{\lambda>0}\llVert \pose{\sum_{k\in\Z} \llvert \tilde{\MS}_{j,0,k}^{\lambda}\pose{f}\rrvert^2}^{1/2}\rrVert_{L^{2}\pose{\R^d\times\R_{+}}}&\lesssim 2^{-j\pose{\frac{d-1}{2}}}\llVert f\rrVert_{L^{2}\pose{\R^d}}.\label{mainL2zero}
\end{align}

 Now, we eliminate the factor $2^{-j\pose{\frac{d-1}{2}}}$ by defining $T^{\lambda}_{j,\ell,k}=2^{j\pose{\frac{d-1}{2}}}\tilde{\MS}_{j,\ell,k}^{\lambda}$, specifically,
\begin{align}\label{Tjkell}
T^{\lambda}_{j,\ell,k}\pose{f}\pose{x,t}&=2^{\frac{k}{2}}\int_{\R^d} e^{2\pi i \left[x\cdot\xi+t\llvert\xi+\lambda Jx\rrvert\right]}\\&\times\chi\pose{2^{k}t}\chi\pose{\frac{t\llvert\xi+\lambda Jx\rrvert}{2^{j}}}\widetilde{\chi}_j\pose{t\pose{\xi+\lambda Jx}}\varphi_{j,\ell}\pose{t^2\lambda}\widehat{f}\pose{\xi}d\xi.\nonumber
\end{align}
Then, the estimates \eqref{mainL2positive} and \eqref{mainL2zero} follow from the estimates
\begin{align}
\sup_{\lambda>0}\sup_{k\in\Z} \llVert T_{j,\ell,k}^{\lambda}\pose{f}\rrVert_{L^{2}\pose{\R^d\times\R_{+}}}&\lesssim 2^{-\ell\pose{\frac{d}{2}}}\llVert f\rrVert_{L^{2}\pose{\R^d}}\ \text{for}\ \ell>0,\label{l>0}
\\
\sup_{\lambda>0}\llVert \pose{\sum_{k\in\Z} \llvert T_{j,0,k}^{\lambda}\pose{f}\rrvert^2}^{1/2}\rrVert_{L^{2}\pose{\R^d\times\R_{+}}}&\lesssim \llVert f\rrVert_{L^{2}\pose{\R^d}}.\label{l<0}
\end{align}

The following dilation invariant property helps us to treat the operator norm of single pieces.
\begin{lemma}[Dilation]\label{dilationop} For every $\pose{j,\ell,k}\in\Z_{+}\times\N_0\times\Z$, we have 
\begin{align}
\llVert T_{j,\ell,k}^{\lambda}\rrVert_{op}&=\llVert T_{j,\ell,0}^{\lambda 2^{-2k}}\rrVert_{op}\label{mainoperatornorm}
\end{align}
where $\llVert \cdot\rrVert_{op}$ denotes the operaotor norm from $L^2\pose{\R^d}$ to $L^2\pose{\R^d\times\R_{+}}$.
\end{lemma}

\begin{proof} Express the phase $\psi^{\lambda}$ and the amplitude $\Phi^{\lambda}$ of $T_{j,\ell,k}^{\lambda}$ in \eqref{Tjkell} as
\begin{align*}
\phi^{\lambda}\pose{x,\xi,t}&=x\cdot\xi+t\llvert\xi+\lambda Jx\rrvert\ \text{and}\\ \Psi^{\lambda}\pose{x,\xi,t}&=\chi\pose{\frac{t\llvert\xi+\lambda Jx\rrvert}{2^{j}}}\widetilde{\chi}_j\pose{t\pose{\xi+\lambda Jx}}\varphi_{j,\ell}\pose{t^2\lambda}.
\end{align*}
Then 
\[T_{j,\ell,k}^{\lambda}\pose{f}\pose{x,t}=2^{\frac{k}{2}}\int_{\R^d} e^{2\pi i \phi^{\lambda}\pose{x,\xi,t}} \chi\pose{2^{k}t}\Psi^{\lambda}\pose{x,\xi,t}\widehat{f}\pose{\xi}d\xi.\]
Using the substitution $\pose{x,\xi,t}\mapsto \pose{2^{-k}x,2^{k}\xi,2^{-k}t}$, we have
\begin{align*}
\phi^{\lambda}\pose{2^{-k}x,2^k\xi,2^{-k}t}&=x\cdot\xi+t\llvert\xi+\lambda2^{-2k} Jx\rrvert=\phi^{\lambda2^{-2k}}\pose{x,\xi,t}\ \text{and}
\\
\Psi^{\lambda}\pose{2^{-k}x,2^k\xi,2^{-k}t}&=\chi\pose{\frac{t\llvert\xi+\lambda2^{-2k} Jx\rrvert}{2^{j}}}\widetilde{\chi}_j\pose{t\llvert\xi+\lambda2^{-2k} Jx\rrvert}\varphi_{j,\ell}\pose{t^2\lambda2^{-2k}}\\&=\Psi^{\lambda2^{-2k}}\pose{x,\xi,t}.
\end{align*}
Hence it holds that
\begin{align*}
T_{j,\ell,k}^{\lambda}\pose{f}\pose{2^{-k}x,2^{-k}t}&=2^{\frac{k}{2}}2^{k\pose{\frac{d}{2}}}\int_{\R^{d}}e^{2\pi i \phi^{\lambda2^{-2k}}\pose{x,\xi,t}}\chi\pose{t}\Psi^{\lambda2^{-2k}}\pose{x,\xi,t}\pose{2^{k\pose{\frac{d}{2}}}\widehat{f}\pose{2^{k}\xi}}d\xi\\
&=2^{\frac{k}{2}}2^{k\pose{\frac{d}{2}}}T_{j,\ell,0}^{\lambda 2^{-2k}}\pose{2^{-k\pose{\frac{d}{2}}}{f}\pose{2^{-k}\cdot}}\pose{x,t}.
\end{align*}
Then, by applying this to the second equation below, we see that 
\begin{align*}
\llVert T_{j,\ell,k}^{\lambda}\pose{f}\rrVert_{L^{2}\pose{\R^d\times\R_{+}}}^2&=\int_{0}^{\infty}\int_{\R^d} \llvert T_{j,\ell,k}^{\lambda}\pose{f}\pose{2^{-k}x,2^{-k}t}\rrvert^2 2^{-kd}2^{-k}dx dt\\
&=\int_{0}^{\infty}\int_{\R^d} \llvert T_{j,\ell,0}^{\lambda 2^{-2k}}\pose{2^{-\frac{d}{2}k}{f}\pose{2^{-k}\cdot}}\pose{x,t}\rrvert^2 dxdt\\&=\llVert T_{j,\ell,0}^{\lambda 2^{-2k}}\pose{2^{-\frac{d}{2}k}{f}\pose{2^{-k}\cdot}}\rrVert_{L^{2}\pose{\R^d\times\R_{+}}}^2.
\end{align*}
Thus, together with $\llVert 2^{-\frac{d}{2}k}{f}\pose{2^{-k}\cdot}\rrVert_{L^{2}\pose{\R^{d}}}=\llVert f\rrVert_{L^{2}\pose{\R^d}}$, we obain \eqref{mainoperatornorm}. 
\end{proof}

\section{$L^2$ Estimates for the Main Terms}\label{sectionl2estimatesformainterms}
 In this section, we shall prove \eqref{l>0} and \eqref{l<0}:
\begin{theorem} Let $\lambda>0$ and let $T_{j,\ell,k}^{\lambda}$ be defined as in \eqref{Tjkell}. Then, it holds that
\begin{align}
\sup_{\lambda>0}\sup_{k\in\Z} \llVert T_{j,\ell,k}^{\lambda}\pose{f}\rrVert_{L^{2}\pose{\R^d\times\R_{+}}}&\lesssim 2^{-\ell\pose{\frac{d}{2}}}\llVert f\rrVert_{L^{2}\pose{\R^d}}\ \text{for}\ \ell>0, \label{l>0sec5}
\\
\sup_{\lambda>0}\llVert \pose{\sum_{k\in\Z} \llvert T_{j,0,k}^{\lambda}\pose{f}\rrvert^2}^{1/2}\rrVert_{L^{2}\pose{\R^d\times\R_{+}}}&\lesssim \llVert f\rrVert_{L^{2}\pose{\R^d}},\label{l=0sec5}
\end{align}
where the implicit constants are independent of $j$ and $\ell$.
\end{theorem}

\subsection{Proof of \eqref{l>0sec5}}
The following proposition helps us to treat both of \eqref{l>0sec5} and \eqref{l=0sec5}.
\begin{proposition}\label{prop3.1} For all positive integers $\ell>0$, it holds that
\begin{equation}
\sup_{\lambda>0}\sup_{k\in\Z}\llVert T_{j,\ell,k}^{\lambda}\rrVert_{op}\leq C2^{-\frac{d}{2}\ell}\label{singlel>0}
\end{equation}
where $C>0$ is a constant independent of $j$ and $\ell$. Furthermore, one can get
\begin{equation}
\sup_{\lambda>0}\sup_{k\in\Z}\llVert T_{j,0,k}^{\lambda}\rrVert_{op}\leq C\label{singlezero}
\end{equation}
where $C>0$ is a constant independent of $j$ and $\ell$.
\end{proposition}

\begin{proof}[\textbf{Proof of \eqref{singlel>0}}] By the dilation property \eqref{mainoperatornorm} in Lemma \ref{dilationop} and by replacing $\lambda$ with $\lambda2^{-2k}$, we have
\begin{equation*}
 \llVert T_{j,\ell,k}^{\lambda}\rrVert_{op}= \llVert T_{j,\ell,0}^{\lambda2^{-2k}}\rrVert_{op}\leq\sup_{\lambda>0}\llVert T_{j,\ell,0}^{\lambda}\rrVert_{op}
\end{equation*}
for any $k\in\Z$. Thus, it suffices to treat the operator norm $\llVert T_{j,\ell,0}^{\lambda}\rrVert_{op}$. Note that $J$ is invertible, $J^{-1}=-J$, and $\llvert \det J\rrvert=1$. By applying the substitution $\pose{\xi,x}\mapsto \pose{2^{j}\xi,\pose{J^{-1}}\frac{2^j}{\lambda}x}$ to $T_{j,\ell,0}^{\lambda}$ in \eqref{Tjkell}, it suffices to consider the operator norm of the following operator $T_{j,\ell,0}^{\lambda}$ defined by
\begin{align*}
T_{j,\ell,0}^{\lambda}\pose{f}\pose{x,t}&=\pose{\frac{2^{2j}}{\lambda}}^{\frac{d}{2}}\\&\times\int_{\R^d} e^{2\pi i \pose{\frac{2^{2j}}{\lambda}}\left[J^{-1}x\cdot\xi+\pose{\frac{\lambda}{2^j}}t\llvert\xi+x\rrvert\right]} \chi\pose{t}\chi\pose{t\llvert\xi+x\rrvert}\widetilde{\chi}_j\pose{2^j t\llvert\xi+x\rrvert}\varphi_{j,\ell}\pose{t^2\lambda}\widehat{f}\pose{\xi}d\xi.
\end{align*}

In order to apply the $T^{*}T$ method, we express the kernel $K_{\ell}\pose{\xi,\eta}$ of $\pose{T_{j,\ell,0}^{\lambda}}^{*}T_{j,\ell,0}^{\lambda}$ as
\begin{equation*}
K_{\ell}\pose{\xi,\eta}=\pose{\frac{2^{2j}}{\lambda}}^{d}\int_{0}^{\infty}\int_{\R^d} e^{2\pi i  \pose{\frac{2^{2j}}{\lambda}}\pose{\phi\pose{x,\xi,t}-\phi\pose{x,\eta,t}}} \Psi\pose{x,\xi,\eta,t}dxdt,
\end{equation*}
where the phase $\phi$ and the amplitude $\Psi$ are given by
\begin{align*}
\phi\pose{x,\xi,t}&=J^{-1}x\cdot\xi+\frac{\lambda}{2^j}t\llvert\xi+x\rrvert\ \text{and} \\\Psi\pose{x,\xi,\eta,t}&=\chi\pose{t}\chi\pose{t\llvert\xi+x\rrvert}\chi\pose{t\llvert\eta+x\rrvert}\widetilde{\chi}_j\pose{2^j t\llvert\xi+x\rrvert}\widetilde{\chi}_j\pose{2^j t\llvert\eta+x\rrvert}\varphi_{j,\ell}\pose{t^2\lambda}.
\end{align*} 
By a unity of participation for $\xi+x$, one can assume that $\llvert\pose{\xi+x}-\pose{\eta+x}\rrvert=\llvert\xi-\eta\rrvert$ is sufficiently small.

 Next, we consider the Hessian matrix $H_{x,t,\xi}\pose{\phi}\pose{x,\xi,t}$ computed by
\[H_{x,t,\xi}\pose{\phi}\pose{x,\xi,t}=\begin{pmatrix}-J^{-1}+\frac{\lambda}{2^j}\frac{t}{\llvert\xi+x\rrvert}I-\frac{\lambda}{2^j}\frac{t}{\llvert\xi+x\rrvert^3}\begin{pmatrix}\pose{\xi_{1}+x_{1}}\pose{\xi+x}\\\vdots\\\pose{\xi_{d}+x_{d}}\pose{\xi+x}\end{pmatrix}\\ \frac{\lambda}{2^j}\frac{\xi+x}{\llvert\xi+x\rrvert} \end{pmatrix}.\]
Let $E_i$ be the $\pose{d+1}\times \pose{d+1}$ row addition matrix by replacing $row_{i}$ with $row_{i}+\frac{t\pose{\xi_i+x_i}}{\llvert\xi+x\rrvert^2}row_{d+1}$ for $i\in\left[d\right]=\left\{1,\cdots,d\right\}$. Then, with $-J^{-1}=J$, we have
\[E_{d}\cdots E_{1}\pose{H_{x,t,\xi}\pose{\phi}\pose{x,\xi,t}}=\begin{pmatrix}J+\frac{\lambda}{2^j}\frac{t}{\llvert\xi+x\rrvert}\\ \frac{\lambda}{2^j}\frac{\xi+x}{\llvert\xi+x\rrvert} \end{pmatrix}.
\]
Using the sub-multiplicative property of the matrix norm, we have 
\begin{equation}\label{3.6}
\llvert\begin{pmatrix}J+\frac{\lambda}{2^j}\frac{t}{\llvert\xi+x\rrvert}\\ \frac{\lambda}{2^j}\frac{\xi+x}{\llvert\xi+x\rrvert}\end{pmatrix}\pose{\xi-\eta}\rrvert\leq\pose{\prod_{i=1}^d\llVert E_i\rrVert_2}\llvert\pose{H_{x,t,\xi}\pose{\phi}\pose{x,\xi,t}}\pose{\xi-\eta}\rrvert
\end{equation}
where $\llVert E_i\rrVert_2$ is the usual matrix norm defined by 
\[\llVert E_i \rrVert_2=\sup_{x\in S^{d-1}}\llvert E_{i}x\rrvert.\]
From the support condition of $\Psi$, we have $1/2\leq t\leq2$ and $1/8\leq \llvert\xi+x\rrvert\leq8$. Thus, for each $i\in\left[d\right]$, it holds that $\frac{t\llvert\xi_i+x_i\rrvert}{\llvert\xi+x\rrvert^2}\leq16$ and $\llVert E_i\rrVert_2\leq16d$. Then, the previous inequality \eqref{3.6} becomes
\begin{equation}\label{3.7}
\llvert\begin{pmatrix}J+\frac{\lambda}{2^j}\frac{t}{\llvert\xi+x\rrvert}\\ \frac{\lambda}{2^j}\frac{\xi+x}{\llvert\xi+x\rrvert} \end{pmatrix}\pose{\xi-\eta}\rrvert\leq C \llvert\pose{H_{x,t,\xi}\pose{\phi}\pose{x,\xi,t}}\pose{\xi-\eta}\rrvert
\end{equation}
where $C >0$ is a constant. Since $J$ is skew-symmetric, we have $\langle Jx, x\rangle=0$ for any $x\in\R^d$. Hence,
\begin{align*}
\llvert\pose{J+\frac{\lambda}{2^j}\frac{t}{\llvert\xi+x\rrvert}I}\pose{\xi-\eta}\rrvert^2&=\llvert J\pose{\xi-\eta}\rrvert^2+\llvert\pose{\frac{\lambda}{2^j}\frac{t}{\llvert\xi+x\rrvert}}\pose{\xi-\eta}\rrvert^2.\nonumber
\end{align*}
By \eqref{3.7}, this implies that
\begin{align}\label{3.8}
3C \llvert\pose{H_{x,t,\xi}\pose{\phi}\pose{x,\xi,t}}\pose{\xi-\eta}\rrvert\geq\llvert J\pose{\xi-\eta}\rrvert+\frac{\lambda}{2^j}\llvert\xi-\eta\rrvert.
\end{align}

On the other hand, one can compute 
\begin{equation}\label{meanvalue}
\nabla_{x,t}\pose{\phi\pose{x,\xi,t}-\phi\pose{x,\eta,t}}=\begin{pmatrix}J\pose{\xi-\eta}+\frac{\lambda}{2^j}t\pose{\frac{\xi+x}{\llvert\xi+x\rrvert}-\frac{\eta+x}{\llvert\eta+x\rrvert}}\\ \frac{\lambda}{2^j}\pose{\llvert\xi+x\rrvert-\llvert\eta+x\rrvert}\end{pmatrix}.
\end{equation}
Then, there exists a constant $C>0$ independent of $x$ and $t$ satisfying
\begin{equation*}
\llvert\nabla_{x,t}\pose{\phi\pose{x,\xi,t}-\phi\pose{x,\eta,t}}-H_{x,t,\xi}\pose{\phi}\pose{x,\xi,t}\pose{\xi-\eta}\rrvert\leq\frac{\lambda}{2^j}C\llvert\xi-\eta\rrvert^2.
\end{equation*}
By combining this with \eqref{3.8}, for sufficiently small $\llvert\xi-\eta\rrvert$, we have
\begin{equation}
\llvert\nabla_{x,t}\pose{\phi\pose{x,\xi,t}-\phi\pose{x,\eta,t}}\rrvert\sim\llvert J\pose{\xi-\eta}\rrvert+\frac{\lambda}{2^j}\llvert\xi-\eta\rrvert.\label{3.10}
\end{equation}

Now, set a vector $b=\nabla_{x,t}\pose{\phi\pose{x,\xi,t}-\phi\pose{x,\eta,t}}$ and define an operator \[\mathcal{L}\left[f\right]=\frac{1}{2\pi i \pose{\frac{2^{2j}}{\lambda}}}\dotprod{\frac{b}{\llvert b\rrvert^2}}{\nabla_{x,t}}\pose{f}.\] Then, it holds that $\mathcal{L}\left[e^{2\pi i \pose{\frac{2^{2j}}{\lambda}}\pose{\phi\pose{x,\xi,t}-\phi\pose{x,\eta,t}}}\right]=e^{2\pi i \pose{\frac{2^{2j}}{\lambda}}\pose{\phi\pose{x,\xi,t}-\phi\pose{x,\eta,t}}}$.
By applying the integration by parts $d+1$ times with respect to $dxdt$ along the direction $\frac{b}{\llvert b\rrvert^2}$, one can see that
\begin{align}\label{integrationbypart}
K_{\ell}\pose{\xi,\eta}&=\pose{\frac{2^{2j}}{\lambda}}^{d}\int_{0}^{\infty}\int_{\R^d} \pose{\mathcal{L}}^{d+1}\left[e^{2\pi i \pose{\frac{2^{2j}}{\lambda}}\pose{\phi\pose{x,\xi,t}-\phi\pose{x,\eta,t}}}\right]\Psi\pose{x,\xi,\eta,t}dxdt\nonumber\\
&=\pose{\frac{2^{2j}}{\lambda}}^{d}\int_{0}^{\infty}\int_{\R^d} e^{2\pi i  \pose{\frac{2^{2j}}{\lambda}}\pose{\phi\pose{x,\xi,t}-\phi\pose{x,\eta,t}}}\pose{\mathcal{L}^t}^{d+1}\left[\Psi\right]\pose{x,\xi,\eta,t}dxdt
\end{align}
where the transpose of $\mathcal{L}$ is given by
\begin{equation*}
\mathcal{L}^{t}\left[\Psi\right]=\pd_{t}\pose{\frac{b_{d+1}\Psi}{2\pi i\pose{\frac{2^{2j}}{\lambda}}\llvert b\rrvert^2}}+\sum_{k=1}^{d}\pd_{x_k}\pose{\frac{b_{k}\Psi}{2\pi i\pose{\frac{2^{2j}}{\lambda}}\llvert b\rrvert^2}}.
\end{equation*}

Now, we shall estimate $\pose{\mathcal{L}^{t}}^{N}\left[\Psi\right]$. First, we utilize \eqref{3.10} to compute the derivatives for ${b_k=\pose{\nabla_{x,t}\pose{\phi\pose{x,\xi,t}-\phi\pose{x,\eta,t}}}_{k}}$ in \eqref{meanvalue}: 
\[\llvert\partial_{x,t}^\beta \pose{b_k}\rrvert\lesssim_{\beta}\frac{\lambda}{2^j}\llvert\xi-\eta\rrvert\leq\llvert J\pose{\xi-\eta}\rrvert+\frac{\lambda}{2^j}\llvert\xi-\eta\rrvert\sim\llvert b\rrvert\]
 for all $k=1,\cdots,d+1$ and for all multi-indices $\beta$ with $\llvert\beta\rrvert\geq1$. Next, by combining this with \eqref{3.10}, we obtain that
\[\llvert\partial_{x,t}^{\beta}\pose{\frac{b}{\llvert b\rrvert^2}}\rrvert\lesssim_{\beta}\frac{1}{\llvert b\rrvert}\sim\frac{1}{\llvert J\pose{\xi-\eta}\rrvert+\frac{\lambda}{2^j}\llvert\xi-\eta\rrvert}\ \ \text{for all}\ \beta.\]
Also, from the support conditions $\llvert\xi+x\rrvert\sim\llvert\eta+x\rrvert\sim t\sim1$, one can see that
\begin{align*}
\llvert\partial_{x,t}^{\beta}\pose{\widetilde{\chi}_j\pose{2^j t\llvert\xi+x\rrvert}}\rrvert=\llvert\int_{\R^d}\partial_{x,t}^{\beta}\pose{\chi\pose{t\pose{\xi+x}-\frac{\eta}{2^j}}}\widehat{\psi}\pose{\eta}d\eta\rrvert\leq C .
\end{align*}
Moreover, from $\frac{t\lambda}{2^{j+\ell}}\lesssim \frac{1}{t}\sim 1$ by the support condition of $\varphi_{j,\ell}\pose{t^2\lambda}$, it holds that
\[\llvert\partial_{t}^{\beta}\pose{\varphi_{j,\ell}\pose{t^2\lambda}}\rrvert \leq 2\pose{\frac{t\lambda}{2^{j+\ell}}+\frac{\lambda}{2^{j+\ell}}}\sim \frac{1}{t}+\frac{1}{t^2}\sim 1.\]
Finally, by utilizing the above inequalities and the support condition $t\sim1$, we deduce that 
\[\llvert\partial_{x,t}^{\beta}\Psi\pose{x,\xi,\eta,t}\rrvert\lesssim_{d,\beta}\tilde{\chi}\pose{t}\tilde{\chi}\pose{t\llvert\xi+x\rrvert}\tilde{\chi}\pose{t\llvert\eta+x\rrvert}\varphi_{j,\ell}\pose{\lambda}.\] Therefore, we have
\begin{equation*}\label{integrationbypartamp}
\llvert\pose{\mathcal{L}^t}^N\left[\Psi\right]\pose{x,\xi,\eta,t}\rrvert\lesssim_{d,N}\frac{\tilde{\chi}\pose{t}\tilde{\chi}\pose{t\llvert\xi+x\rrvert}\tilde{\chi}\pose{t\llvert\eta+x\rrvert}\varphi_{j,\ell}\pose{\lambda}}{\pose{1+\pose{\frac{2^{2j}}{\lambda}}\pose{\llvert J\pose{\xi-\eta}\rrvert+\frac{\lambda}{2^j}\llvert\xi-\eta\rrvert}}^{N}}
\end{equation*}
for all $N\in\N$. Putting this into \eqref{integrationbypart} with $\varphi_{j,\ell}\pose{\lambda}=\chi\pose{\frac{\lambda}{2^{j+\ell}}}$ for $\ell>0$, we obtain 
\begin{align}\label{3.11}
\llvert K_{\ell}\pose{\xi,\eta}\rrvert&=\pose{\frac{2^{2j}}{\lambda}}^{d}\int_{0}^{\infty}\int_{\R^d} \llvert\pose{\mathcal{L}^t}^{d+1}\left[\Psi\right]\pose{x,\xi,\eta,t}\rrvert dxdt\\
&\lesssim\pose{\frac{2^{2j}}{\lambda}}^{d}\frac{\chi\pose{\frac{\lambda}{2^{j+\ell}}}}{\pose{1+\pose{\frac{2^{2j}}{\lambda}}\pose{\llvert J\pose{\xi-\eta}\rrvert+\frac{\lambda}{2^j}\llvert\xi-\eta\rrvert}}^{d+1}}\nonumber.
\end{align}
Then, by discarding $\llvert J\pose{\xi-\eta}\rrvert$, we see that
\begin{align*}
\int_{\R^d}\llvert K_{\ell}\pose{\xi,\eta}\rrvert d\xi&\lesssim\pose{\frac{2^{2j}}{\lambda}}^{d}\int_{\R^d} \frac{\chi\pose{\frac{\lambda}{2^{j+\ell}}}}{\pose{1+2^j\llvert\xi-\eta\rrvert}^{d+1}} d\xi \nonumber\\
&\lesssim \pose{\frac{2^j}{\lambda}}^d\chi\pose{\frac{\lambda}{2^{j+\ell}}}\lesssim 2^{-d\ell}.\nonumber
\end{align*}
Similarly, we obtain $\int_{\R^d}\llvert K_{\ell}\pose{\xi,\eta}\rrvert d\eta\lesssim2^{-d\ell}$. By Schur's test, the inequality \eqref{singlel>0} holds.
\end{proof}
\begin{proof}[\textbf{Proof of \eqref{singlezero}}]
On the other hand, by noting that $\llvert J\pose{\xi-\eta}\rrvert=\llvert\xi-\eta\rrvert$ and that $\varphi_{j,0}\pose{\lambda}=\psi\pose{\frac{\lambda}{2^{j}}}\leq1$, we directly obtain from \eqref{3.11} that
\begin{align*}
\int_{\R^d}\llvert K_{\ell}\pose{\xi,\eta}\rrvert d\xi&\lesssim \pose{\frac{2^{2j}}{\lambda}}^{d}\int_{\R^d} \frac{1}{\pose{1+\frac{2^{2j}}{\lambda}\llvert \xi-\eta\rrvert}^{d+1}} d\xi\lesssim 1.\nonumber
\end{align*}
Similarly, we have $\int_{\R^d}\llvert K_{\ell}\pose{\xi,\eta}\rrvert d\eta\lesssim1$. Again, by Schur's test, \eqref{singlezero} holds.
\end{proof}

\subsection{Proof of \eqref{l=0sec5}}
Now, there remains \eqref{l=0sec5}. We first state an almost orthogonality lemma.
\begin{lemma}[Cotlar-Stein]\label{lem3.1}Let $M_1, M_2$ be two measure spaces and let $\left\{T_k\right\}_{k\in\Z}$ be a family of the operators $T_k : L^2\pose{M_1}\mapsto L^2\pose{M_2}.$ If $\llVert T_{k_1}T^*_{k_2}\rrVert_{op}\leq C\pose{\llvert k_1-k_2\rrvert}$ with $\sum_{k\in\Z}\sqrt{C\pose{\llvert k\rrvert}}$ being finite, then for all $f\in L^2\pose{M_1}$,
\[\llVert\pose{\sum_{k\in\Z}\llvert T_k f\rrvert^2}^{1/2}\rrVert_{L^2\pose{M_2}}\lesssim\pose{\sup_{k}\llVert T_k\rrVert_{op}\sum_{k\in\Z}\sqrt{C\pose{\llvert k\rrvert}}}^{1/2}\llVert f\rrVert_{L^2\pose{M_2}}.\]
\end{lemma}
\begin{proof}
See \cite[Lemma~3.2]{seeger}.
\end{proof}
Using Lemma \ref{lem3.1} and \eqref{singlezero} in Proposition \ref{prop3.1}, one can obtain \eqref{l=0sec5} by the next proposition.
\begin{proposition}\label{prop3.2} For $\pose{j,k_1,k_2}\in\Z_{+}\times\Z\times\Z$, it holds that
\begin{equation}\label{3.15}
\llVert T_{j,0,k_{1}}^{\lambda}\pose{T_{j,0,k_{2}}^{\lambda}}^{*}\rrVert_{op}\leq C2^{-\llvert k_{1}-k_{2}\rrvert}
\end{equation}
where $C>0$ is independent of $j$ and $\lambda$.
\end{proposition}
\begin{proof} Note that $J$ is invertible. By applying the substitution $\pose{\xi,x}\mapsto\pose{\lambda\xi,J^{-1}x}$ to $T_{j,0,k}^{\lambda}$ in \eqref{Tjkell}, it suffices to consider a new operator $T_{j,0,k}^{\lambda}$ defined by
\begin{align*}
T_{j,0,k}^{\lambda}\pose{f}\pose{x,t}&=\lambda^{\frac{d}{2}}2^{\frac{k}{2}}\int_{\R^d} e^{2\pi i \lambda\left[J^{-1}x\cdot\xi+t\llvert\xi+x\rrvert\right]} \\
&\times\chi\pose{2^{k}t}\chi\pose{\frac{t\lambda\llvert\xi+x\rrvert}{2^j}}\widetilde{\chi}_j\pose{t\lambda\pose{\xi+x}}\psi\pose{\frac{t^2\lambda}{2^{j}}}\widehat{f}\pose{\xi}d\xi.
\end{align*}
For estimating $\llVert T_{j,0,k_{1}}^{\lambda}\pose{T_{j,0,k_{2}}^{\lambda}}^{*}\rrVert_{op}$, one can compute the kernel $K_{1,2}$ of $T_{j,0,k_1}^{\lambda}\pose{T_{j,0,k_2}^{\lambda}}^{*}$ as
\begin{align*}
K_{1,2}\pose{x,y,t,s}&=\lambda^{d}2^{\pose{\frac{k_1+k_2}{2}}}\int_{\R^d}e^{2\pi i \lambda\Phi\pose{x,y,\xi,t,s}}\Psi\pose{x,y,\xi,t,s}d\xi
\end{align*}
where  the phase function $\phi\pose{x,\xi,t}$ of $T_{j,0,k}^{\lambda}\pose{f}\pose{x,t}$ is defined by 
\[\phi\pose{x,\xi,t}=J^{-1}x\cdot\xi+t\llvert\xi+x\rrvert,\] 
and hence the phase function $\Phi\pose{x,y,\xi,t,s}$ of $K_{1,2}\pose{x,y,t,s}$ is defined by
\begin{equation*}
\Phi\pose{x,y,\xi,t,s}=\phi\pose{x,\xi,t}-\phi\pose{y,\xi,t}=J^{-1}\pose{x-y}\cdot\xi+t\llvert\xi+x\rrvert-s\llvert\xi+y\rrvert,
\end{equation*}
and the amplitude $\Psi\pose{x,y,\xi,t,s}$ is given by
\begin{align*}
\Psi\pose{x,y,\xi,t,s}&=\chi\pose{2^{k_{1}}t}\chi\pose{\frac{t\lambda\llvert\xi+x\rrvert}{2^{j}}}\widetilde{\chi}_j\pose{t\lambda\pose{\xi+x}}\psi\pose{\frac{t^2\lambda}{2^{j}}}\\&\times\chi\pose{2^{k_{2}}s}\chi\pose{\frac{s\lambda\llvert\xi+y\rrvert}{2^{j}}}\widetilde{\chi}_j\pose{s\lambda\pose{\xi+y}}\psi\pose{\frac{s^2\lambda}{2^{j}}}.
\end{align*}

We shall prove \eqref{3.15} for a fixed $\lambda$. From the support condition of $\Psi$, one can see that 
\begin{equation}
t\sim2^{-k_1},\ s\sim2^{-k_2}, \ \text{and}\ \lambda\leq 4\min\{2^{j+2k_1}, \ 2^{j+2k_2}\}.\label{squaresupport1}
\end{equation}
For this fixed $\lambda$, we temporarily set $\ell_1, \ell_2$ to be integers such that
\begin{align}\label{squaresupport2}
2^{\ell_1}=\lambda2^{-j-2k_1}\ \text{and}\ 2^{\ell_2}=\lambda2^{-j-2k_2}.
\end{align}
Then, together with the support condition of $\Psi$ and \eqref{squaresupport1}, we deduce that
\begin{equation}
\llvert\xi+x\rrvert\sim2^{-k_1-\ell_1}\ \text{and}\ \llvert\xi+y\rrvert\sim2^{-k_2-\ell_2}.\label{supportstt*}
\end{equation}
Without loss of generality, assume that $k_1 - k_2 > 10$. Then, by the relation 
\[
2\pose{k_{1}-k_{2}}=\ell_{2}-\ell_{1}
\]
(which is derived from \eqref{squaresupport2}), it follows that $\ell_{2} > \ell_{1} + 10$. Moreover, combining this with \eqref{supportstt*} and the fact that $\ell_1, \ell_2 \le 2$ (as implied by \eqref{squaresupport1} and \eqref{squaresupport2}), we obtain
\begin{equation}
t\sim2^{-k_1}\ll s\sim2^{-k_2}\lesssim \llvert\xi+y\rrvert\sim2^{-k_2-\ell_2}\ll \llvert\xi+x\rrvert\sim2^{-k_1-\ell_1}. \label{comparison}
\end{equation} 
Here, we write $A\ll B$ to mean that $2^{4}A\leq B$.

Note that 
\begin{equation}\label{maingradtt}
\nabla_{\xi}\Phi\pose{x,y,\xi,t,s}=J^{-1}\pose{x-y}+t\frac{\xi+x}{\llvert\xi+x\rrvert}-s\frac{\xi+y}{\llvert\xi+y\rrvert}.
\end{equation}
Also, it holds that $\llvert\pose{x-y}\rrvert=\llvert\pose{\xi+x}-\pose{\xi+y}\rrvert\sim2^{-k_1-\ell_1}$  from \eqref{comparison}. Thus, we have
\begin{equation}
\llvert J^{-1}\pose{x-y}\rrvert=\llvert x-y\rrvert\sim2^{-k_1-\ell_1}. \label{domiphase}
\end{equation}
Note that, from \eqref{comparison},
\begin{align*}
\llvert t\frac{\xi+x}{\llvert\xi+x\rrvert}-s\frac{\xi+y}{\llvert\xi+y\rrvert}\rrvert\leq t+s\lesssim 2^{-k_2}\ll2^{-k_1-\ell_1}.
\end{align*}
Together with \eqref{maingradtt} and \eqref{domiphase}, we obtain
\begin{equation}\label{Phase}
\llvert\nabla_{\xi}\Phi\pose{x,y,\xi,t,s}\rrvert \sim \llvert J^{-1}\pose{x-y}\rrvert\sim2^{-k_1-\ell_1}.
\end{equation}

On the other hand, from the support condition of $\psi\pose{\frac{t\lambda}{2^j}}$ and $t\sim2^{-k_1}$, one can see that
\begin{align*}
\llvert\pd_{\xi}^{\beta}\pose{\widetilde{\chi}_j\pose{t\lambda\pose{\xi+x}}}\rrvert=\llvert\int_{\R^d}\pd_{\xi}^{\beta}\pose{\chi\pose{\frac{t\lambda\pose{\xi+x}-\eta}{2^j}}}\widehat{\psi}\pose{\eta}d\eta\rrvert\leq C \pose{\frac{t\lambda}{2^j}}^{\llvert\beta\rrvert}
\end{align*}
and similarly $\llvert\pd_{\xi}^{\beta}\pose{\widetilde{\chi}_j\pose{s\lambda\pose{\xi+y}}}\rrvert\lesssim\pose{\frac{s\lambda}{2^j}}^{\llvert\beta\rrvert}$ for all multi-indices $\beta$. Then, we obtain 
\begin{align}\label{gradpsi}
\llvert\pd_{\xi}^{\beta}\Psi\pose{x,y,\xi,t,s}\rrvert\lesssim\frac{1}{2^{\pose{-k_2-\ell_2}\llvert\beta\rrvert}}\chi\pose{2^{k_{1}}t}\chi\pose{2^{k_{2}}s}\chi\pose{\frac{t\lambda\llvert\xi+x\rrvert}{2^{j}}}\chi\pose{\frac{s\lambda\llvert\xi+y\rrvert}{2^{j}}}
\end{align} 
from $\frac{t\lambda}{2^j}\sim2^{k_1+\ell_1}\ll2^{k_2+\ell_2}\sim\frac{s\lambda}{2^j}$ in \eqref{squaresupport2} and \eqref{comparison} combined with $\psi\pose{\frac{t^2\lambda}{2^{j}}}\leq1$ and slightly modifications of $\chi$.

Take $b=\nabla_{\xi}\Phi\pose{x,y,\xi,t,s}$ and define an operator $\mathcal{L}$ by $\mathcal{L}\left[f\right]=\frac{1}{2\pi i\lambda}\dotprod{\frac{b}{\llvert b\rrvert^2}}{\nabla_{\xi}}\pose{f}$. Then it holds that $\mathcal{L}\left[e^{2\pi i\lambda\Phi}\right]=e^{2\pi i\lambda\Phi}$. By applying the integration by parts $d+2$ times with direction $\frac{b}{\llvert b\rrvert^2}$, we obtain that 
\begin{align}\label{intparttt}
K_{1,2}\pose{x,y,t,s}&=\lambda^{d}2^{\pose{\frac{k_1+k_2}{2}}}\int_{\R^d}\pose{\mathcal{L}}^{d+2}\left[e^{2\pi i \lambda\Phi\pose{x,y,\xi,t,s}}\right]\Psi\pose{x,y,\xi,t,s}d\xi\\
&=\lambda^{d}2^{\pose{\frac{k_1+k_2}{2}}}\int_{\R^d}e^{2\pi i \lambda\Phi\pose{x,y,\xi,t,s}}\pose{\mathcal{L}^t}^{d+2}\left[\Psi\right]\pose{x,y,\xi,t,s}d\xi\nonumber
\end{align}
where the transpose of $\mathcal{L}$ is defined by
\[\mathcal{L}^{t}\left[\Psi\right]=\sum_{n=1}^{d}\pd_{\xi_n}\pose{\frac{b_{n}\Psi}{2\pi i\lambda\llvert b\rrvert^2}}.\]
From \eqref{comparison} and \eqref{maingradtt}, one can observe that
\begin{equation}\label{gradgradphase}
\llvert\partial_{\xi}^{\beta}b_{n}\rrvert=\llvert\partial_{\xi}^{\beta}\pd_{\xi_n}\Phi\pose{x,y,\xi,t,s}\rrvert\lesssim_{\beta}\frac{2^{-k_2}}{2^{\pose{-k_2-\ell_2}\llvert\beta\rrvert}}
\end{equation}
for all multi-indices $\beta$ with $\llvert\beta\rrvert\geq1$ and for all $n=1,\cdots,d$. Also, from \eqref{comparison} and \eqref{Phase}, it holds that $\frac{2^{-k_2}}{\llvert b\rrvert}\lesssim\frac{2^{-k_2}}{2^{-k_1-\ell_1}}\leq1$. Using this with \eqref{gradgradphase}, one can compute
\begin{equation*}
\llvert\pd_{\xi}^{\beta}\pose{\frac{b}{\llvert b\rrvert^2}}\rrvert\lesssim \frac{1}{\llvert b\rrvert2^{\pose{-k_2-\ell_2}\llvert \beta\rrvert}}\ \text{for all multi-indices}\ \beta.
\end{equation*}
Together with \eqref{gradpsi}, we obtain
\begin{equation*}\label{ttttt}
\llvert\pose{\mathcal{L}^{t}}^{d+2}\left[\Psi\right]\pose{x,y,\xi,t,s}\rrvert\lesssim\frac{\llvert\chi\pose{2^{k_{1}}t}\chi\pose{2^{k_{2}}s}\chi\pose{\frac{t\lambda\llvert\xi+x\rrvert}{2^{j}}}\chi\pose{\frac{s\lambda\llvert\xi+y\rrvert}{2^{j}}}\rrvert}{\pose{\lambda2^{-k_2-\ell_2}\llvert J^{-1}\pose{x-y}\rrvert+1}^{d+2}}.
\end{equation*}
Now, by using this in \eqref{intparttt} and measuring the support size $2^{\pose{-k_2-\ell_2}d}$ of $\chi\pose{\frac{s\lambda\llvert\xi+y\rrvert}{2^{j}}}$, we have
\begin{align}\label{kernelesti}
\llvert K_{1,2}\pose{x,y,t,s}\rrvert&\lesssim \lambda^d 2^{\pose{\frac{k_{1}+k_{2}}{2}}}\frac{2^{\pose{-k_2-\ell_2}d}\chi\pose{2^{k_{1}}t}\chi\pose{2^{k_{2}}s}}{\pose{\lambda2^{-k_2-\ell_2}\llvert J^{-1}\pose{x-y}\rrvert+1}^{d+2}}.
\end{align} 
Among the $d+2$ factors in the denominator on the left-hand side in \eqref{kernelesti}, we use one factor—together with \eqref{domiphase} and the choice $\lambda=2^{j+2k_1+\ell_1}$ from \eqref{squaresupport2}—to obtain a $2^{-\pose{k_1-k_2}}$ decay rate:
\begin{equation*}\label{sigleker}
\llvert\frac{1}{\pose{\lambda2^{-k_2-\ell_2}\llvert J^{-1}\pose{x-y}\rrvert+1}}\rrvert \lesssim \frac{1}{2^{j+2k_1+\ell_1}2^{-k_2-\ell_2}2^{-k_1-\ell_1}}\leq2^{-\pose{k_1-k_2}}.
\end{equation*}
Along with this, we also utilize the other $d+1$ decay factors in \eqref{kernelesti} and apply $\int dt\sim2^{-k_1}$ in the following estimate:
\begin{align*}
\int_{0}^{\infty}\int_{\R^d}\vert K_{1,2}(x,y,t,s)\vert dxdt&\lesssim \lambda^d  2^{\pose{\frac{k_{1}+k_{2}}{2}}}\int_{0}^{\infty}\int_{\R^d}\frac{2^{\pose{-k_2-\ell_2}d}\chi\pose{2^{k_{1}}t}\chi\pose{2^{k_{2}}s}}{\pose{\lambda2^{-k_2-\ell_2}\llvert J^{-1}\pose{x-y}\rrvert+1}^{d+2}}dxdt\nonumber\\
&\lesssim\lambda^d  2^{\pose{\frac{k_{2}-k_1}{2}}}
\int_{\R^d}\frac{2^{\pose{-k_2-\ell_2}d}2^{-\pose{k_1-k_2}}}{\pose{\lambda2^{-k_2-\ell_2}\llvert J^{-1}\pose{x-y}\rrvert+1}^{d+1}}dx\nonumber\\
&\lesssim2^{\pose{\frac{k_2-k_1}{2}}}2^{-\pose{k_1-k_2}}=2^{-\frac{3}{2}\pose{k_1-k_2}}.\label{shurs11}
\end{align*}
Similarly, by applying $\int ds\sim 2^{-k_2}$ instead, we obtain
\begin{align*}
\int_{0}^{\infty}\int_{\R^d}\llvert K_{1,2}\pose{x,y,t,s}\rrvert dyds\lesssim2^{\pose{\frac{k_1-k_2}{2}}}2^{-\pose{k_1-k_2}}=2^{-\frac{1}{2}\pose{k_1-k_2}}.
\end{align*}
Then, by Schur's test, we have $\llVert T_{j,0,k_{1}}^{\lambda}\pose{T_{j,0,k_{2}}^{\lambda}}^{*}\rrVert_{op}\lesssim 2^{-(k_{1}-k_{2})}$.
\end{proof}

\section{Appendix}\label{appendix}

\subsection{Proof of $L^p$-boundedness of $\MM_1$ and $\MM_{j}^{O}$} 
We shall show that
\begin{align}
\llVert\MM_{1}\rrVert_{L^{p}\rightarrow L^{p}}&\lesssim1\ \text{for } 1<p\leq\infty, 
\\ 
\llVert\MM_{j}^{O}\rrVert_{L^{p}\rightarrow L^{p}}&\lesssim2^{-jd}\  \text{for all}\ 1<p\leq\infty\ \text{and all}\ j\geq 2. \label{6.2}
\end{align}
Utilizing the rapidly decreasing property of $\sm_j$ and $\sm_{j}^{O}$ in \eqref{measureofsmj} and \eqref{outdecreasing}, one can see that
 \begin{align*}
 \llvert \sm_{1}\pose{x}\rrvert \leq \frac{C }{\pose{1+\llvert x\rrvert}^{d+1}}\ \text{and} \
 \llvert \sm_{j}^{O}\pose{x} \rrvert \leq \frac{C 2^{-jd}}{\pose{1+\llvert x\rrvert}^{d+1}}.
 \end{align*}
 Then, together with $\MA_{j}\pose{f}\pose{\ux,t}=f\ast_{\bH}\pose{\sm_{j}}_t\pose{\ux}$ in \eqref{averageall}, we deduce that
 \begin{align*}
  \llvert \MA_{j}\pose{f}\pose{x,x_{d+1},t}\rrvert&\leq \frac{C }{t^{d}}\int_{\R^d} \llvert f\pose{x-y,x_{d+1}-\langle Jx,y\rangle}\rrvert \frac{1}{\pose{1+\llvert \frac{y}{t} \rrvert}^{d+1}} dy, \\
 \llvert \MA_{j}^{O}\pose{f}\pose{x,x_{d+1},t}\rrvert&\leq \frac{C 2^{-jd}}{t^{d}}\int_{\R^d} \llvert f\pose{x-y,x_{d+1}-\langle Jx,y\rangle}\rrvert \frac{1}{\pose{1+\llvert \frac{y}{t} \rrvert}^{d+1}} dy.
 \end{align*}
 Thus, both of $\pose{6.1}$ and \eqref{6.2} follow from the below lemma:
\begin{lemma}\label{prop6.1}
Define a maximal operator $M\pose{f}$ by
\begin{equation}
M\pose{f}\pose{x,x_{d+1}}=\sup_{t>0}\frac{1}{t^d}\int_{\R^d} \llvert f\pose{x-y,x_{d+1}-\langle Jx,y\rangle}\rrvert \frac{1}{\pose{1+\llvert \frac{y}{t} \rrvert}^{d+1}} dy.
\end{equation}
Then we have $\llVert M\pose{f}\rrVert_{L^p\pose{\R^{d+1}}}\leq C \llVert f\rrVert_{L^p\pose{\R^{d+1}}}$ for all $1<p\leq\infty$.
\end{lemma}
\begin{proof}
See \cite[Corollary~2.5]{folland1982hardy}.
\end{proof}
 
\subsection{Proof of $L^2$ Estimates \eqref{l2tail} for the Tail Terms}
In this subsection, we prove \eqref{l2tail}:
\[\llVert \MV_{j,\ell}\rrVert_{L^2\pose{\R^{d+1}}\rightarrow L^2\pose{\R^{d+1}}}\leq C  2^{-j\pose{\frac{d-2}{2}}}2^{-\ell\pose{\frac{d}{2}}}.\]
Recall the form of the tail maximal function $\MV_{j,\ell}$ in \eqref{insideandoutside} and \eqref{outside}:
\begin{align*}
\widehat{d\nu_j}\pose{\xi}&=\chi_0^c\pose{\frac{\xi}{2^j}}\widehat{\sm_{j}^{I}}\pose{\xi}=\chi_0^c\pose{\frac{\xi}{2^j}}\widetilde{\chi}_j\pose{\xi}\widehat{\sm}\pose{\xi},
\\
\MS_{j,\ell}^c\pose{f}\pose{\ux,t}&=\int_{\R^{d+1}} e^{2\pi i\ux\cdot\uxi}\widehat{d\nu_{j}}\pose{t\pose{\xi+\xi_{d+1}Jx}}\varphi_{j,\ell}\pose{t^2\xi_{d+1}}\widehat{f}\pose{\xi,\xi_{d+1}}d\xi d\xi_{d+1}, \\
\MV_{j,\ell}\pose{f}\pose{\ux}&=\sup_{t>0}\llvert \MS_{j,\ell}^c\pose{f}\pose{\ux,t}\rrvert.
\end{align*}
We first prove \eqref{l2tail} in the case of $\ell=0$. 
\begin{proof}[\textbf{Proof of the case of $\ell=0$}] One can observe that
\begin{align}
\MS_{j,0}^c\pose{f}\pose{\ux,t}=\int_{\R^{d+1}}f\pose{\uy^{-1}\cdot\ux}\label{l=0tail}\frac{1}{t^d}d\nu_j\pose{\frac{y}{t}}\frac{2^j}{t^2}\psi^{\vee}\pose{\frac{2^j}{t^2}y_{d+1}}dydy_{d+1}.
\end{align}
Using the integration by parts, we have
\begin{align*}
\llvert d\nu_j\pose{y}\rrvert\leq\llvert\int_{\R^d}\int_{\R^d}\int_{\R^d}e^{2\pi i \pose{y-x}\cdot\xi}\chi_0^c\pose{\frac{\xi}{2^j}}\chi\pose{\frac{\xi-\eta}{2^j}}\widehat{\psi}\pose{\eta}d\xi d\eta\sm\pose{x}\rrvert\leq\frac{C_N2^{-jN}}{\pose{1+\llvert y\rrvert}^{N}}.
\end{align*}
for any sufficiently large $N>100d$.
Thus, from \eqref{l=0tail}, one can see that
\begin{align*}
\llvert\MS_{j,0}^c\pose{f}\pose{\ux,t}\rrvert&\leq\frac{2^j}{t^{d+2}}\int_{\R^{d+1}}\llvert f\pose{\uy^{-1}\cdot\ux}\rrvert\llvert d\nu_j\pose{\frac{y}{t}}\psi^{\vee}\pose{\frac{2^j}{t^2}y_{d+1}}\rrvert dydy_{d+1} \\
&\lesssim_{N}\frac{2^{-jN}}{t^{d+2}}\int_{\R^{d+1}}\llvert f\pose{\uy^{-1}\cdot\ux}\rrvert \frac{1}{\pose{1+\llvert \frac{y}{t}\rrvert}^{N}}\frac{1}{\pose{1+\llvert\frac{2^j y_{d+1}}{t^2}\rrvert}^2}dydy_{d+1}. 
\end{align*}
The maximal function of the last term in the above inequality can be treated by the similar way in the proof of Lemma \ref{weak11R}. Thus we have  \eqref{l2tail} in the case of $\ell=0$. 
\end{proof}

Now, it remains to proving the case of $\ell\geq1$.
\begin{proof}[\textbf{Proof of the case of $\ell\geq1$}] Similarly to the reduction for $L^2$ estimates of the main terms $ M_{j,\ell}$ in Section \ref{subsectionsobolev}, we begin with the standard Sobolev embedding argument. One can see that
\begin{equation*}
\llvert\MV_{j,\ell}\pose{f}\pose{\ux}\rrvert\leq\pose{\sum_{k\in\Z}\sup_{t>0}\llvert\MS_{j,\ell,k}^c\pose{f}\pose{\ux,t}\rrvert^2}^{\frac{1}{2}}
\end{equation*}
where $\MS_{j,\ell,k}^c\pose{f}$ is defined as
\begin{align}
\MS_{j,\ell,k}^c\pose{f}\pose{\ux,t}&=\chi(2^k t)\MS_{j,\ell}^c\pose{f}\pose{\ux,t}\nonumber\\
&=\chi(2^k t)\int_{\R^{d+1}} e^{2\pi i\ux\cdot\uxi}\widehat{d\nu_{j}}\pose{t\pose{\xi+\xi_{d+1}Jx}}\chi\pose{\frac{t^2\xi_{d+1}}{2^{j+\ell}}}\widehat{f}\pose{\xi,\xi_{d+1}}d\xi d\xi_{d+1}\label{tailsobolevterm}
\end{align}
for $k\in\Z$. Using the standard Sobolev embedding argument, we have
\begin{align}\label{squaresumsobolevsec6}
\int_{\R^{d+1}}\llvert\MV_{j,\ell}\pose{f}\pose{\ux}\rrvert^2 d\ux\leq 2\sum_{k\in\Z}\int_{\R^{d+1}}\int_{0}^{\infty}\llvert \pd_{t}\MS_{j,\ell,k}^c\pose{f}\pose{\ux,t}\rrvert \llvert\MS_{j,\ell,k}^c\pose{f}\pose{\ux,t}\rrvert dtd\ux.
\end{align} 
We treat the $t$ derivative of the amplitude $\chi(2^k t)\widehat{d\nu_{j}}\pose{t\pose{\xi+\xi_{d+1}Jx}}\chi\pose{\frac{t^2\xi_{d+1}}{2^{j+\ell}}}$ in \eqref{tailsobolevterm} by the same way in Subsection \ref{subsectionsobolev} with the only difference $\llvert\pd_{t}\left[\chi_0^c\pose{\frac{t\xi}{2^j}}\right]\rrvert\lesssim\frac{\llvert\xi\rrvert}{2^j}\chi_0^c\pose{\frac{t\xi}{2^j}}$.
Then, with $t\sim2^{-k}$ and $\llvert\frac{t\xi_{d+1}}{2^{j+\ell}}\rrvert\sim 2^{k}$ in \eqref{tailsobolevterm}, one can see that
\begin{align*}
\pd_{t}\pose{\chi(2^k t)\widehat{d\nu_{j}}\pose{t\pose{\xi+\xi_{d+1}Jx}}\chi\pose{\frac{t^2\xi_{d+1}}{2^{j+\ell}}}}\lesssim 2^k \chi(2^k t)&\llvert\widehat{d\nu_{j}}\pose{t\pose{\xi+\xi_{d+1}Jx}}\rrvert\chi\pose{\frac{t^2\xi_{d+1}}{2^{j+\ell}}}\\
+\llvert \xi+\xi_{d+1}Jx\rrvert\chi(2^k t)&\llvert\widehat{d\nu_{j}}\pose{t\pose{\xi+\xi_{d+1}Jx}}\rrvert\chi\pose{\frac{t^2\xi_{d+1}}{2^{j+\ell}}}.
\end{align*}
with slight modifications to $\widehat{\sm}$ and $\chi$. Also, from the support conditions of $\chi_0^c\pose{\frac{t\xi}{2^j}}$ and $\chi\pose{\frac{t\xi-\eta}{2^j}}$, one can see that $\llvert \eta\rrvert\geq\frac{1}{10}\pose{\llvert t\xi\rrvert+\llvert \eta\rrvert}$ and $\llvert \eta\rrvert\geq\frac{2^j}{2}$. Thus we have
\begin{align*}
\llvert \chi_0^c\pose{\frac{t\xi}{2^j}}\pose{\int_{\R^d}\chi\pose{\frac{t\xi-\eta}{2^j}}\widehat{\psi}\pose{\eta}d\eta}\rrvert&\leq \frac{2^{-3jd}C }{\pose{1+\llvert t\xi\rrvert}^{d+2}},
\\
\llvert \widehat{d\nu_j}\pose{t\xi}\rrvert=\llvert\chi_0^c\pose{\frac{t\xi}{2^j}}\pose{\int_{\R^d}\chi\pose{\frac{t\xi-\eta}{2^j}}\widehat{\psi}\pose{\eta}d\eta}\widehat{\sm}\pose{t\llvert \xi\rrvert}\rrvert&\lesssim\frac{2^{-3jd}}{\pose{1+\llvert t\xi\rrvert}^{d+2}} \ \text{and so}
\\
\llvert \xi+\xi_{d+1}Jx\rrvert\chi(2^k t)\llvert\widehat{d\nu_{j}}\pose{t\pose{\xi+\xi_{d+1}Jx}}\rrvert&\lesssim\frac{2^{-3jd}2^k\chi\pose{2^k t}}{\pose{1+t\llvert \xi+\xi_{d+1}Jx\rrvert}^{d+1}}. 
\end{align*}
Thus, with \eqref{tailsobolevterm}, we have the following two pointwise inequality:
\begin{align*}
\llvert\pd_{t} \MS_{j,\ell,k}^c\pose{f}\pose{\ux,t}\rrvert&\lesssim2^{-3jd}2^k\llvert W_{j,\ell,k}\pose{f}\pose{\ux,t}\rrvert\\
\llvert\MS_{j,\ell,k}^c\pose{f}\pose{\ux,t}\rrvert&\lesssim2^{-3jd}\llvert W_{j,\ell,k}\pose{f}\pose{\ux,t}\rrvert
\end{align*}
where $W_{j,\ell,k}\pose{f}$ is defined by
\begin{equation*}
W_{j,\ell,k}\pose{f}\pose{\ux,t}=\int_{\R^d} \frac{\chi\pose{2^k t}}{\pose{1+t\llvert \xi+\xi_{d+1}Jx\rrvert}^{d+1}}\chi\pose{\frac{t^2\xi_{d+1}}{2^{j+\ell}}}\llvert\widehat{f}\pose{\xi,\xi_{d+1}}\rrvert d\xi d\xi_{d+1}.
\end{equation*}
This implies a Sobolev embedding in $t$ such that
\begin{align}\label{tailsobolev}
\sup_{t>0}\llvert\MS_{j,\ell,k}^c\pose{f}\pose{\ux,t}\rrvert^2&\leq 2\int_{0}^{\infty}\llvert\pd_{t}\MS_{j,\ell,k}^c\pose{f}\pose{x,t}\rrvert\llvert \MS_{j,\ell,k}^c\pose{f}\pose{\ux,t}\rrvert dt\\
&\lesssim_d 2^{-6jd}2^k\int_{0}^{\infty}\llvert W_{j,\ell,k}\pose{f}\pose{\ux,t}\rrvert^2 dt.\nonumber
\end{align}
By \eqref{squaresumsobolevsec6} and \eqref{tailsobolev}, we have 
\begin{align*}
\llVert \MV_{j,\ell}\pose{f}\rrVert_{L^{2}\pose{\R^{d+1}}}\lesssim 2^{-3jd}\llVert \pose{\sum_{k\in\Z} \llvert 2^{\frac{k}{2}}W_{j,\ell,k}\pose{f}\rrvert^2}^{1/2} \rrVert_{L^{2}\pose{\R^{d+1}\times \R_{+}}}.
\end{align*}
Now, we shall show that
\begin{align*}
2^{-3jd}\llVert \pose{\sum_{k\in\Z} \llvert 2^{\frac{k}{2}}W_{j,\ell,k}\pose{f}\rrvert^2}^{1/2} \rrVert_{L^{2}\pose{\R^{d+1}\times \R_{+}}}\lesssim 2^{-j\pose{\frac{d-1}{2}}}2^{-\ell\pose{\frac{d}{2}}}\llVert f\rrVert_{L^2 \pose{\R^{d+1}}}.
\end{align*}
Utilizing the same method in the proof of Propositon \ref{reductionuniformL2}, it suffices to show that 
\begin{align}
\sup_{\lambda>0}\llVert \pose{\sum_{k\in\Z} \llvert 2^{\frac{k}{2}}W_{j,\ell,k}^{\lambda}\pose{f}\rrvert^2}^{1/2} \rrVert_{L^{2}\pose{\R^{d}\times \R_{+}}}\lesssim 2^{-\ell\pose{\frac{d}{2}}}\llVert f\rrVert_{L^2 \pose{\R^{d}}}\label{aaaaa2}
\end{align}
where $W_{j,\ell,k}^{\lambda}\pose{f}$ is defined by
\begin{align}
W_{j,\ell,k}^{\lambda}\pose{f}\pose{x,t}&=\int_{\R^d} \frac{\chi\pose{2^k t}}{\pose{1+t\llvert \xi+\lambda Jx\rrvert}^{d+1}}\chi\pose{\frac{t^2\lambda }{2^{j+\ell}}}\llvert\widehat{f}\pose{\xi}\rrvert d\xi.\label{aaaaa5}
\end{align}
Since we consider the $L^2$ operator norm, by utilizing Plancherel's theorem, we consider the following operator instead of \eqref{aaaaa5}:
\[W_{j,\ell,k}^{\lambda}\pose{f}\pose{x,t}=\int_{\R^d} \frac{\chi\pose{2^k t}}{\pose{1+t\llvert \xi+\lambda Jx\rrvert}^{d+1}}\chi\pose{\frac{t^2\lambda }{2^{j+\ell}}}\llvert f\pose{\xi}\rrvert d\xi.\]
For a fixed $\lambda>0$, $j\geq1$ and $\ell\geq1$, there are only finite numbers of $k$ satisfying the support condition of $\chi\pose{2^k t}\chi\pose{\frac{t^2\lambda}{2^{j+\ell}}}$. Thus, to show \eqref{aaaaa2}, it suffices to show that
\begin{align}
\sup_{\lambda>0}\sup_{k\in\Z}\llVert 2^{\frac{k}{2}}W_{j,\ell,k}^{\lambda}\rrVert_{op}&\leq C  2^{-\ell\pose{\frac{d}{2}}}.\label{aaaaa4}
\end{align}
Since the dilation structures of $\tilde{\MS}_{j,\ell,k}^{\lambda}$ in Lemma \ref{dilationop} and $2^{\frac{k}{2}}W_{j,\ell,k}^{\lambda}$ work in the same way, we derive the following relation by the same method in Lemma \ref{dilationop}:
\begin{align*}
\llVert 2^{\frac{k}{2}}W_{j,\ell,k}^{\lambda}\rrVert_{op}=\llVert W_{j,\ell,0}^{\lambda2^{-2k}}\rrVert_{op}.
\end{align*}
By the above dilation property, one can see that
\[\sup_{\lambda>0}\sup_{k\in\Z}\llVert 2^{\frac{k}{2}}W_{j,\ell,k}^{\lambda}\rrVert_{op}\leq\sup_{\lambda>0}\llVert W_{j,\ell,0}^{\lambda}\rrVert_{op}.\]
We finish by using Schur's test for $\llVert W_{j,\ell,0}^{\lambda}\rrVert_{op}$:
\begin{align*}
\int_{0}^{\infty}\int_{\R^d}\frac{\chi\pose{t}}{\pose{1+t\llvert \xi+\lambda Jx\rrvert}^{d+1}}\chi\pose{\frac{t^2\lambda }{2^{j+\ell}}}dxdt\leq C \lambda^{-d}\chi\pose{\frac{\lambda }{2^{j+\ell}}}&\lesssim2^{-jd-d\ell},\\
\int_{\R^d}\frac{\chi\pose{t}}{\pose{1+t\llvert \xi+\lambda Jx\rrvert}^{d+1}}\chi\pose{\frac{t^2\lambda }{2^{j+\ell}}}d\xi&\leq C .
\end{align*}
Thus, by Schur's test, we have $\llVert W_{j,\ell,0}^{\lambda}\rrVert_{op}\lesssim2^{-\ell\pose{\frac{d}{2}}}$ and the desired estimate \eqref{aaaaa4} holds.
\end{proof}

\newpage
\bibliographystyle{plain}
\bibliography{ref98}

\end{document}